\documentclass[a4paper,12pt]{amsart}

\usepackage{amsmath,amssymb,amsthm,graphicx,pstricks,comment,accents}

\usepackage[all]{xy}

\usepackage{lscape}
\usepackage{hyperref}
\usepackage{tikz}
\usetikzlibrary{arrows}
\usetikzlibrary{decorations}
\usetikzlibrary{shapes}

\author{Claire Amiot}
\address{Institut Fourier, 100 rue des maths, 38402 Saint Martin d'H\`eres, Université Grenoble Alpes, Institut Universitaire de France}
\email{claire.amiot@univ-grenoble-alpes.fr}

\newtheorem{theorem}{Theorem}[section]

\newtheorem{lemma}[theorem]{Lemma}
\newtheorem{corollary}[theorem]{Corollary}
\newtheorem{proposition}[theorem]{Proposition}
\theoremstyle{remark}
\newtheorem{remark}[theorem]{Remark}

\theoremstyle{definition}
\newtheorem{definition}[theorem]{Definition}

\definecolor{dark-green}{RGB}{14,150,2}
\newcommand{\gpoint}{\color{dark-green}{\circ}}
\newcommand{\rpoint}{{\red \bullet}}
\newcommand{\cross}{{\sf \red x}}
\newcommand{\bZ}{\mathbb Z}
\newcommand{\grading}{\mathbf{n}}

\newcommand{\surf}{\mathcal S}
\newcommand{\cD}{\mathcal{D}}
\DeclareMathOperator{\proj}{proj}
\DeclareMathOperator{\gr}{gr}

\title[Derived categories of skew-gentle algebras and orbifolds]{Indecomposable objects in the derived category of a skew-gentle algebra using orbifolds}

\thanks{The author is supported by the French ANR grant CHARMS (ANR-19-CE40-0017)and by the Institut Universitaire de France}

\begin{document}
\maketitle

\begin{abstract} Skew-gentle algebras over a field of characteristic $\neq 2$ are skew-group algebras
 of certain gentle algebras endowed with a $\bZ_2$-action. Using the topological description of Opper, Plamondon and Schroll in \cite{OpperPlamondonSchroll} for the indecomposable objects of the derived category of any gentle algebra, one obtains here a complete description of indecomposable objects in the derived category of any skew-gentle algebra in terms of curves on an orbifold surface.The results presented here are complementary to the ones in \cite{LabardiniSchrollValdivieso}. First, we obtain a complete classification of indecomposable objects and not of ``homotopy strings'' and ``homotopy bands'' which are not always indecomposable. Second, the classification obtained here does not use the combinatorial description of \cite{BekkertMarcosMerklen}, but topological arguments coming from the double cover of the orbifold surface constructed in \cite{AmiotBrustle}. 
 
  \end{abstract}
\tableofcontents

\section{Introduction}

Gentle algebras were introduced in the 80's by Assem and Skowronski as a generalisation of tilted algebras of type $\mathbb A$ and $\widetilde{\mathbb A}$. These algebras are tame and derived tame, and the indecomposable objects of their derived category have ben described explicitely in combinatorial terms in \cite{BekkertMerklen} and \cite{BurbanDrozd}. More recently, a topological approach has been used to describe their representation theory in \cite{OpperPlamondonSchroll, BaurCoelho} (see also \cite{HaidenKatzarkovKontsevich}). More precisely, a marked surface together with a collection of arcs have been attached to any gentle algebra. Using this data, Opper, Plamondon and Schroll have obtained in \cite{OpperPlamondonSchroll} a description of indecomposable objects of the derived category of a gentle algebra in terms of graded curves on the corresponding surface. This new description has been obtained by translating the combinatorial description in topological terms. 

Skew-gentle algebras, introduced in the 90's by Geiss and de la Pe\~na in \cite{GeissdelaPena} can both be seen as a generalisation of gentle algebras, and as skew-group gentle algebras. They are also derived tame \cite{BurbanDrozd}. Using their description as skew-group algebras, one can attach to each skew-gentle algebra $A$, a gentle algebra $\widetilde{A}$ with an action of $\mathbb Z_2$. Therefore, using the topological description in \cite{OpperPlamondonSchroll}, one associates in \cite{AmiotBrustle} to each skew-gentle algebra $A$, a marked surface $\widetilde{\surf}$ with a collection of arcs invariant under the action of a homeomorphism $\sigma$ of order $2$. It becomes then natural to associate to $A$ the orbifold surface $\surf:=\widetilde{S}/\sigma$. The aim of this paper is to use this topological model to obtain a complete description of indecomposable objects of $\cD^{\rm b}(A)$ in terms of graded curves on the orbifold $\surf$. More precisely the main result of this paper is the following :

\begin{theorem}
Let $k$ be an algebraically closed field of characteristic $\neq 2$. Let $(\surf, M_\rpoint,P_\rpoint,X_\cross,D)$ be a dissected marked orbifold. Let $A$ be the skew-gentle $k$-algebra attached to it. Then the indecomposable objects of the category $\cD^{\rm b}(A)$ are in bijection with the following five sets :
\begin{enumerate}
\item[$(\surf 1)$] $\left\{(\gamma,\grading)\in \pi_1^{\rm orb,gr}(\surf,M_{\gpoint}, P_{\rpoint})\ |\ \gamma\neq \gamma^{-1}\right\}/\sim$ where $(\gamma,\grading)\sim (\gamma^{-1},\grading)$;
\item[$(\surf 2)$] $\left\{ (\gamma,\grading,\epsilon)\in \pi_1^{\rm orb,gr}(\surf,M_{\gpoint}, P_{\rpoint})\times \{\pm 1\}\ |\ \gamma=\gamma^{-1}\right\}$

\item [$(\surf 3)$]$\{ ([\gamma],\mathbf{n},\lambda)\in \pi_1^{\rm orb,free, gr}(\surf)\times k^*\ |\ [\gamma]\neq [\gamma^{-1}]\}/\sim$ where $([\gamma],\mathbf{n},\lambda)\sim ([\gamma^{-1}],\mathbf{n},\lambda^{-1})$;
\item[$(\surf 4)$] $\{ ([\gamma],\mathbf{n},\lambda)\in \pi_1^{\rm orb,free, gr}(\surf)\times k^*\setminus\{\pm 1\}\ |\ [\gamma]= [\gamma^{-1}], \gamma^2\neq 0 \}/\sim$;
\item[$(\surf 5)$] $\{ ([\gamma],\mathbf{n},\epsilon,\epsilon')\in \pi_1^{\rm orb,free, gr}(\surf)\times \{\pm 1\}^2\ |\ [\gamma]= [\gamma^{-1}], \gamma^2\neq 0\}.$
\end{enumerate}
where the set $\pi_1^{\rm orb,gr}(\surf,M_{\gpoint},P_\rpoint)$ is defined in subsection \ref{section4.1}, and where the set $\pi_1^{\rm orb,free, gr}(\surf)$ is defined in subsection \ref{section4.2}.
\end{theorem}
 Note that using a similar topological model, a description of certain objects has been already obtained in \cite{LabardiniSchrollValdivieso}. Translating the combinatorial description of \cite{BekkertMarcosMerklen}  of the derived category of a skew-gentle algebra, the authors obtain a topological description of `homotopy strings' and `homotopy bands' in $\cD^{\rm b}(A)$. The description is however quite different, since homotopy strings and homotopy bands are sometimes indecomposable objects, and sometimes not. Furthermore the techniques used in the present paper are completely different since one does not use  Bekkert-Marcos-Merklen's description. The strategy is as follows. We first use results of Reiten and Riedtmann \cite{ReitenRiedtmann} on skew-group algebras in order to obtain a description of indecomposable objects of $\cD^{\rm b}(A)$ in terms of those of $\cD^{\rm b}(\widetilde{A})$ where $\widetilde{A}$ is the $\bZ_2$-gentle algebra attached to $A$. We then use the topological description of \cite{OpperPlamondonSchroll} to obtain a description of indecomposables of $\cD^{\rm b}(A)$ in terms of the surface $\widetilde{\surf}$. Finally, we use topological arguments on fundamental groups of an orbifold surface to obtain a description in terms of the orbifold $\surf$. Similar arguments have been used in \cite{AmiotPlamondon} to get a complete description of indecomposable objects of the cluster category associated with a marked surface with punctures.

\medskip

The plan of the paper is as follows. In Section 2, we recall results of \cite{OpperPlamondonSchroll, AmiotBrustle} relating a dissected surface (resp. dissected orbifold) with a gentle (resp. skew-gentle) algebra. We then recall the topological description of the indecomposable objects of $\cD^{\rm b}(\widetilde{A})$ due to \cite{OpperPlamondonSchroll}, and deduce a description of the objects in $\cD^{\rm b}(A)$ in term of the two-folded cover $\widetilde{\surf}$. In section 4, we introduce orbifold fundamental groups, to obtain a description of the indecomposable in terms of curves in $\surf$. We end the paper by a detailed example.

\subsection*{Acknowledgement}
The author would like to thank Thomas Br\"ustle for interesting discussions on this project.

\section{Skew-gentle algebras and orbifolds}

Let $k$ be a field.

We start the paper by recalling the topological model attached to the derived category of gentle and skew-gentle algebras.

 \subsection{Skew-gentle algebras and $\cross$-dissections}

\begin{definition}
A {\em gentle pair} is a pair $(Q,I)$ given by a quiver $Q$ and  a subset $I$ of paths of length 2 in $Q$ such that 
\begin{itemize}
\item for each $i\in Q_0$, there are at most two arrows with source $i$, and at most two arrows with target $i$;
\item for each arrow $\alpha:i\to j$ in $Q_1$, there exists at most one arrow $\beta$ with target $i$ such that $\beta\alpha\in I$ and at most one arrow $\beta'$ with target $i$ such that $\beta'\alpha\notin I$;

\item for each arrow $\alpha:i\to j$ in $Q_1$, there exists at most one arrow $\beta$ with source $j$ such that $\alpha\beta\in I$ and at most one arrow $\beta'$ with source $j$ such that $\alpha\beta'\notin I$.
\item the algebra $A(Q,I):=kQ/I$ is finite dimensional.
\end{itemize}
A $k$-algebra is {\em gentle} if it admits a presentation $A=kQ/I$ where $(Q,I)$ is a gentle pair.
\end{definition}

 \begin{definition}
A \emph{skew-gentle triple} $(Q,I,{\rm Sp})$ is the data of  a quiver $Q$,  a subset $I$ of paths of length two in $Q$, and  a subset ${\rm Sp}$ of loops in $Q$ (called `\emph{special loops}') such that $(Q, I\amalg  \{e^2, e\in {\rm Sp}\})$ is a gentle pair. 
In this case, the algebra $$A(Q,I,{\rm Sp}):= kQ/\langle I\amalg \{e^2-e,e\in {\rm Sp}\}\rangle,$$ is  called a  {\em skew-gentle algebra.}
\end{definition}

A \emph{$\cross$-marked surface} $(\surf, M_\rpoint,P_\rpoint, X_\cross)$ is the data of  
\begin{itemize}
\item an orientable closed smooth surface $\surf$ with non empty boundary;
\item a finite set of marked points $M_\rpoint$ on the boundary, such that there is at least one marked point on each boundary component;
\item two finite sets $P_\rpoint$ and $X_\cross$ of marked points in the interior of $\surf$.

\end{itemize} 
A curve on the boundary of $\surf$ intersecting marked points only on its endpoints is called a \emph{boundary segment}.

An \emph{arc} on $(\surf, M_\rpoint, P_\rpoint)$ is a non contractible curve $\gamma:[0,1]\to \surf$ such that $\gamma_{|(0,1)}$ is injective and $\gamma(0)$ and $\gamma(1)$ are points in $M_\rpoint\cup P_\rpoint\cup X_\cross$. Each arc is considered up to isotopy (fixing endpoints).

\begin{definition}\label{def dissection}
A \emph{$\cross$-dissection} is a collection $D=\{\gamma_1,\ldots,\gamma_s\}$ of arcs cutting $\surf$ into polygons with exactly one side being a boundary segment, and such that each $\cross$ in $X_\cross$ is the endpoint of exactly one arc in $D$.

Two dissected surfaces $(\surf,M_\rpoint,P_\rpoint,X_\cross,D)$ and $(\surf',M'_\rpoint,P'_\rpoint,X'_\cross, D')$ are called homeomorphic if there exists an orientation preserving homeomorphism $\Phi:\surf\to \surf'$ such that $\Phi(M_\rpoint)=M'_\rpoint$, $\Phi(P_\rpoint)=P'_\rpoint$, $\Phi(X_\cross)=X'_\cross$ and $\Phi(D)=D'$.
\end{definition} 

 Following \cite{OpperPlamondonSchroll} and \cite{AmiotBrustle} (see also \cite{BaurCoelho} and \cite{LabardiniSchrollValdivieso}), one can associate to the dissection $D$ a skew-gentle triple $(Q,I,{\rm Sp})$, and thus a skew-gentle algebra $A(D):=A(Q,I,{\rm Sp})$.

\begin{itemize}
\item
The vertices of $Q$ are in bijection with the arcs of $D$;
\item
Given $i$ and $j$ arcs in $D$, there is one arrow
  \scalebox{0.8}{
  \begin{tikzpicture}
  \node (I) at (0,0) {$i$};
  \node (J) at (2,0) {$j$};
  \draw[thick, ->] (I)--node[fill=white, inner sep=0pt]{$\alpha$}(J);
  \end{tikzpicture}} in $Q$ whenever the arcs $i$ and $j$ have a common endpoint in $M_\rpoint\cup P_\rpoint\cup X_\cross$ and when $i$ is immediately followed by the arc $j$ in the clockwise order around their common endpoint (note that if a marked point is the endpoint of exactly one arc, we obtain a loop in the quiver);
  
   \item The set of special loops ${\rm Sp}$ corresponds to the loops $i\to i$ where $i$ is the unique arc ending in each point $\cross$ in $X_\cross$. 
  
 \item If $i$, $j$, and $k$ have a common endpoint in $M_\rpoint\cup P_\rpoint$, and are consecutive arcs following the clockwise  order around $\rpoint$, then we have $\beta\alpha\in I$, where $\alpha$ (resp. $\beta$) is the arrow corresponding to to the angle $j\to i$ (resp. $k\to j$);

 \end{itemize}
 
 We refer to Subsection \ref{section5.1} for an example.
 
 \begin{proposition}\cite{AmiotBrustle,LabardiniSchrollValdivieso}
 The assignement $D\to A(D)$ maps $\cross$-dissections to skew-gentle algebras, and all skew-gentle algebras are obtained in this way. If moreover, $A(D)$ and $A(D')$ are isomorphic skew-gentle algebras which are not gentle, then the dissected surfaces $(\surf,M_\rpoint, P_\rpoint, X_\cross,D)$ and $(\surf',M'_\rpoint, P'_\rpoint, X'_\cross,D')$ are homeomorphic.
 \end{proposition}

\subsection{Two folded covering and $\bZ_2$-actions}
The name skew-gentle introduced by Geiss and de la Pe\~na in \cite{GeissdelaPena} comes from the fact that over a field of characteristic $\neq 2$ any skew-gentle algebra is Morita equivalent to the skew-group algebra of a gentle algebra. Recall that if $\Lambda$ is a $k$-algebra, and $G$ a group acting on $\Lambda$ by automorphism, the algebra $\Lambda G$ is defined as follows

\begin{itemize}
\item $\Lambda G:=\Lambda\otimes _k kG$ as $k$-vector space;
\item the multiplication is  given by the formula $$(\lambda\otimes g).(\lambda'\otimes g'):= \lambda g(\lambda')\otimes gg'\textrm{ for all }\lambda,\lambda'\in \Lambda\textrm{ and } g,g'\in G$$
extended by bilinearity.
\end{itemize}

In the case where the order of $G$ is invertible in $k$, then the representation theory of $\Lambda G$ is closely related with the one of $\Lambda$ (see \cite{ReitenRiedtmann}). From now we assume that the characteristic of $k$ is $\neq 2$.

\begin{proposition}\label{prop::skewgentle=skewgroup}\cite{GeissdelaPena,AmiotBrustle}
Let $(Q,I,{\rm Sp})$ be a skew-gentle triple. There is an action of $\bZ_2$ on $A(Q,I,{\rm Sp})$ sending $\epsilon$ to $(1-\epsilon)$ for each $\epsilon$ in ${\rm Sp}$. With this action, the skewgroup algebra $A\bZ_2$ is Morita equivalent to a gentle algebra $\widetilde{A}=A(\widetilde{Q},\widetilde{I})$. 
\end{proposition}

Moreover, if the skew-gentle algebra $A$ is given by a $\cross$-dissected surface $(\surf,M_\rpoint,P_\rpoint,X_\cross,D)$, one can understand geometrically the dissected surface $(\widetilde{\surf},\widetilde{M}_\rpoint,\widetilde{P}_\rpoint,\widetilde{D})$ associated with the gentle algebra $\widetilde{A}$. The construction is given in \cite{AmiotBrustle}.
We recall it here for the convenience of the reader : we fix a point on each boundary segment.  This set of points is denoted by $M_{\gpoint}$ and will be refered as ``green'' points in the sequel.

Then for each $X_i\in X_\cross$, we draw a curve $\gamma_i$ from $X_i$ to the green point which is on the boundary of the polygon containing $X_i$ in such a way that the $\gamma_i$'s do not intersect and stay in the interior of the polygon. We cut the surface $\surf$ along the curves $\gamma_i$ to obtain a surface $\surf^+$. We take another copy $\surf^-$ of this new surface $\surf^+$, and we glue $\surf^+$ and $\surf^-$ along the green segment (see picture below). 
 \[\scalebox{0.6}
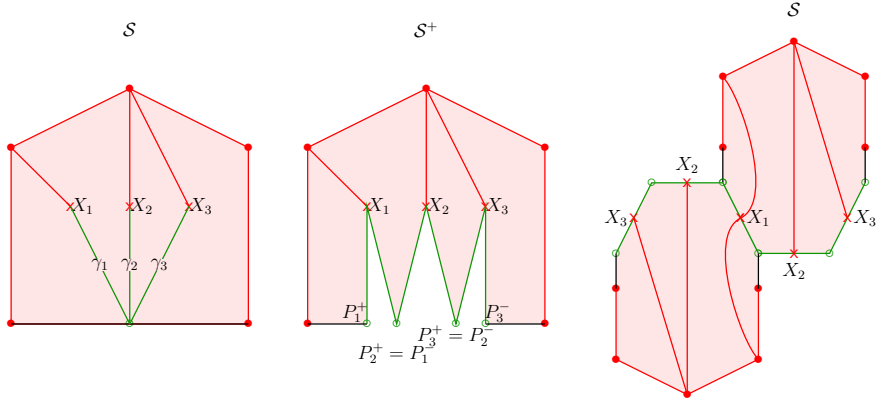
{
  \begin{tikzpicture}[scale=1.3,>=stealth]
  
  \node at (2,5) {$\surf$};
  \draw[thick, red, fill=red!10] (0,0)--(0,3)--(2,4)--(4,3)--(4,0)--(0,0);
\node at (0,0) {$\rpoint$};
\node at (0,3) {$\rpoint$};
\node at (2,4) {$\rpoint$};
\node at (4,3) {$\rpoint$};
\node at (4,0) {$\rpoint$};
\node at (2,0) {$\gpoint$};
\node at (1,2) {$\cross$};
\node at (2,2) {$\cross$};
\node at (3,2) {$\cross$};
\node at (2.2,2) {$X_2$};
\node at (1.2,2) {$X_1$};
\node at (3.2,2) {$X_3$};

\draw[thick, red] (0,3)--(1,2);
\draw[thick, red] (2,2)--(2,4)--(3,2);

\draw[thick] (0,0)--(4,0);
\draw[thick, dark-green] (2,0)--node [black, fill=red!10, inner sep=0pt]{$\gamma_1$}(1,2);
\draw[thick, dark-green] (2,0)--node [black, fill=red!10, inner sep=0pt]{$\gamma_2$}(2,2);
\draw[thick, dark-green] (2,0)--node [black, fill=red!10, inner sep=0pt]{$\gamma_3$}(3,2);

\begin{scope}[xshift=5cm]

\node at (2,5) {$\surf^+$};

 \draw[white, fill=red!10] (0,0)--(0,3)--(2,4)--(4,3)--(4,0)--(3,0)--(3,2)--(2.5,0)--(2,2)--(1.5,0)--(1,2)--(1,0)--(0,0);
 
\node at (0,0) {$\rpoint$};
\node at (0,3) {$\rpoint$};
\node at (2,4) {$\rpoint$};
\node at (4,3) {$\rpoint$};
\node at (4,0) {$\rpoint$};
\node at (1,0) {$\gpoint$};
\node at (1.5,0) {$\gpoint$};
\node at (2.5,0) {$\gpoint$};
\node at (3,0) {$\gpoint$};

\node at (2.2,2) {$X_2$};
\node at (1.2,2) {$X_1$};
\node at (3.2,2) {$X_3$};

\node at (1,2) {$\cross$};
\node at (2,2) {$\cross$};
\node at (3,2) {$\cross$};

\draw[thick, red] (0,0)--(0,3)--(2,4)--(4,3)--(4,0);

\draw[thick, red] (0,3)--(1,2);
\draw[thick, red] (2,2)--(2,4)--(3,2);

\draw[thick] (0,0)--(1,0);
\draw[thick] (3,0)--(4,0);

\draw[thick, dark-green] (1,0)--(1,2)--(1.5,0)--(2,2)--(2.5,0)--(3,2)--(3,0);

\node at (0.8,0.2) {$P_1^+$};
\node at (1.5,-0.5) {$P_2^+=P_1^-$};
\node at (2.5,-0.2) {$P_3^+=P_2^-$};
\node at (3.2,0.2) {$P_3^-$};

\end{scope}

\begin{scope}[xshift=12cm, yshift=3cm,scale=0.6]

\node at (2,4) {$\widetilde{\surf}$};
\draw[white, fill=red!10] (0,-1)--(0,2)--(2,3)--(4,2)--(4,-1)--(3,-3)--(1,-3)--(0,-1);

\node at (0,0) {$\rpoint$};
\node at (0,2) {$\rpoint$};
\node at (2,3) {$\rpoint$};
\node at (4,2) {$\rpoint$};
\node at (4,0) {$\rpoint$};

\node at (0,-1) {$\gpoint$};
\node at (1,-3) {$\gpoint$};
\node at (3,-3) {$\gpoint$};
\node at (4,-1) {$\gpoint$};

\draw[thick, red] (0,0)--(0,2)--(2,3)--(4,2)--(4,0);
\draw[thick,dark-green](0,-1)--(1,-3)--(3,-3)--(4,-1);
\draw[thick] (0,0)--(0,-1);
\draw[thick] (4,0)--(4,-1);

\node at (0.5,-2) {$\cross$};
\node at (2,-3) {$\cross$};
\node at (3.5,-2) {$\cross$};
\node at (1,-2) {$X_1$};
\node at (2,-3.5) {$X_2$};
\node at (4,-2) {$X_3$};

\draw[red, thick] (0.5,-2)..controls (1.5,-1.5) and (0.5,1.5)..(0,2);
\draw[red,thick] (2,-3)--(2,3);
\draw[red, thick] (3.5,-2)--(2,3);

\begin{scope}[ xshift=1cm, yshift=-4cm, rotate=180]
\draw[white, fill=red!10] (0,-1)--(0,2)--(2,3)--(4,2)--(4,-1)--(3,-3)--(1,-3)--(0,-1);

\node at (0,0) {$\rpoint$};
\node at (0,2) {$\rpoint$};
\node at (2,3) {$\rpoint$};
\node at (4,2) {$\rpoint$};
\node at (4,0) {$\rpoint$};

\node at (0,-1) {$\gpoint$};
\node at (1,-3) {$\gpoint$};
\node at (3,-3) {$\gpoint$};
\node at (4,-1) {$\gpoint$};

\draw[thick, red] (0,0)--(0,2)--(2,3)--(4,2)--(4,0);
\draw[thick,dark-green](0,-1)--(1,-3)--(3,-3)--(4,-1);
\draw[thick] (0,0)--(0,-1);
\draw[thick] (4,0)--(4,-1);

\node at (0.5,-2) {$\cross$};
\node at (2,-3) {$\cross$};
\node at (3.5,-2) {$\cross$};
\node at (2,-3.5) {$X_2$};
\node at (4,-2) {$X_3$};

\draw[red, thick] (0.5,-2)..controls (1.5,-1.5) and (0.5,1.5)..(0,2);
\draw[red,thick] (2,-3)--(2,3);
\draw[red, thick] (3.5,-2)--(2,3);

\end{scope}
 \end{scope}   
  
  \end{tikzpicture}}\]

We obtain this way an oriented smooth surface $\widetilde{\surf}$ with boundary, which comes naturally with a diffeomorphism $\sigma$ of order $2$ exchanging the copies $\surf^+$ and $\surf^-$ (see Subsection \ref{section5.1} for an example). The following is shown in \cite[Theorem 4.6]{AmiotBrustle}

\begin{theorem}
Let $(\surf,M_\rpoint,P_\rpoint,X_\cross,D)$ be a $\cross$-dissected surface, and $A:=A(D)$ the associated skew-gentle algebra. Then there exists a surface $(\widetilde{\surf},\widetilde{M}_\rpoint,\widetilde{P}_\rpoint)$ together with an orientation preserving diffeomorphism $\sigma$ of order $2$, fixing globally $\widetilde{M}_\rpoint$ and $\widetilde{P}_{\rpoint}$ such that 

\begin{enumerate}
\item there exists a $2$-folded cover $p:\widetilde{\surf}\to \surf$ branched in the points in $X_\cross$ that induces a diffeomorphism $(\widetilde{S}\setminus \widetilde{X})/\sigma\to \surf\setminus X$ where $\widetilde{X}=p^{-1}(X)$ are the points fixed by $\sigma$;

\item $\widetilde{D}:=p^{-1}(D)$ is a dissection $\sigma$-invariant of  $(\widetilde{\surf},\widetilde{M}_\rpoint,\widetilde{P}_\rpoint)$;

\item the gentle algebra $\widetilde{A}:=A(\widetilde{D})$ is the one in Proposition \ref{prop::skewgentle=skewgroup}, so is Morita equivalent to the skewgroup algebra $A\bZ_2$;

\item the action of $\sigma$ on $\widetilde{D}$ induces an action of $\bZ_2$ on the gentle algebra $\widetilde{A}$ and the algebras $\widetilde{A}\bZ_2$ and $A$ are Morita equivalent.

\end{enumerate}

\end{theorem} 

As a consequence of item $(4)$ together with \cite{ReitenRiedtmann} we obtain the following

\begin{corollary}\label{cor::objectsAobjectsAtilde}
The $\bZ_2$-action on $\widetilde{A}:=A(\widetilde{D})$ induces a $\bZ_2$-action on the indecomposable objects of $\cD^{\rm b}(\widetilde{A})$, and we have a bijection between the isomorphism classes of indecomposable objects in $\cD^{\rm b}(A)$ with the union of the following two sets
\begin{itemize}
\item $\{\sigma\textrm{-invariant indec. in }\cD^{\rm b}(\widetilde{A})\}/\textrm{isom.}\times \bZ_2 $
\item $\{\sigma\textrm{-orbits of non }\sigma\textrm{-invariant indec. in }\cD^{\rm b}(\widetilde{A})\}/\textrm{isom.}.$
\end{itemize}
\end{corollary}

\section{Description of indecomposable objects}

In this section, and in the rest of the paper $k$ is an algebraically closed field of characteristic $\neq 2$. Let $(\surf,M_\rpoint,P_\rpoint,X_\cross,D)$ be a $\cross$-dissected surface and  $A:=A(D)$ be the associated skew-gentle algebra. We fix a set $M_{\gpoint}$ of  green points lying in each boundary segment of the marked surface $(\surf,M_\rpoint)$ as in the previous section. We denote by $(\widetilde{\surf},\widetilde{M}_\rpoint,\widetilde{P}_\rpoint, \widetilde{D},\sigma)$ the associated $\bZ_2$-surface, by $\widetilde{A}:=A(\widetilde{D})$ the associated gentle algebra, and by $\widetilde{M}_{\gpoint}$ the preimage of $M_{\gpoint}$ under the projection $\widetilde{\surf}\to \surf$.

\subsection{Indecomposables in $\cD^{\rm b}(A)$}

We start by recalling results in \cite{OpperPlamondonSchroll} where they describe indecomposable objects of $\cD^{\rm b}(\widetilde{A})$ in terms of curves on $\widetilde{\surf}$.

\medskip

\begin{definition}\label{def::pi_1}
A \emph{curve} on $\widetilde{\surf}\setminus \widetilde{P}_{\rpoint}$ with endpoints in $\widetilde{M}_{\gpoint}$ and $\widetilde{P}_{\rpoint}$ consists of the following three classes of regular curves:
 
 \begin{itemize}
 \item either a regular curve $\gamma:[0,1]\to \widetilde{\surf}\setminus \widetilde{P}_\rpoint$ with endpoints $\gamma(0)$ and $\gamma(1)$ in $\widetilde{M}_{\gpoint}$,
 \item or an infinite regular curve $\gamma:[0,1)\to \widetilde{\surf}\setminus \widetilde{P}_\rpoint$ (resp. $\gamma:(0,1]\to \widetilde{\surf}\setminus \widetilde{P}_\rpoint$) with $\gamma(0)$ (resp. $\gamma(1)$ ) in $\widetilde{M}_{\gpoint}$ and whose infinite ray circles around  a point in $\widetilde{P}_\rpoint$ in counter clockwise orientation; 
 
 \item or an infinite regular curve  $\gamma:(0,1)\to \widetilde{\surf}\setminus \widetilde{P}_\rpoint$ whose infinite rays circles around  points in $\widetilde{P}_\rpoint$ in counter clockwise orientation.
 \end{itemize} 
 
 We consider these curves up to homotopy.
For finite curves,  the homotopy relation is regular homotopy preserving endpoints. For infinite curves, we consider the following homotopy equivalence relation : two curves $\gamma$ and $\gamma'$ with $\gamma(0)=\gamma'(0)$ in $\widetilde{M}_{\gpoint}$ and whose infinite ray circles around the same point $p$ in $\widetilde{P}_\rpoint$ are \emph{homotopic} if there exists a basis $\mathcal{B}$ of neighbourhoods of $p$ such that for any $V\in \mathcal{B}$, $\gamma$ and $\gamma'$ restricted to $\widetilde{\surf}\setminus V$ are connected, and their restrictions $\gamma_V,\gamma_V':[0,1]\to \widetilde{\surf}\setminus V$ are regular homotopic with respect  to their endpoints $\gamma_V(0)=\gamma'_V(0)$ and $\gamma_V(1),\gamma_V'(1)\in \partial V$. Similar homotopy is defined for infinite curves wrapping around points in $\widetilde{P}_{\rpoint}$ on both ends.
 
 We denote by $\pi_1(\widetilde{\surf},\widetilde{M}_{\gpoint}, \widetilde{P}_{\rpoint})$ the set of curves up to homotopy.
 
 \end{definition}

 \begin{remark} 
Note that as stated in Remark 1.20 of \cite{OpperPlamondonSchroll}, infinite curves as above can be considered as finite curves with endpoints in $\widetilde{P}_\rpoint$. But it is however not clear to define a groupoid structure on $\pi_1(\widetilde{\surf},\widetilde{M}_{\gpoint}, \widetilde{P}_{\rpoint})$ if the set $\widetilde{P}_{\rpoint}$ is non empty. 
\end{remark}

We then associate to the dissection $\widetilde{D}$ a \emph{line field} $\widetilde{\eta}_{\widetilde{D}}$ on $\widetilde{\surf}\setminus (\partial \widetilde{S}\cup \widetilde{P}_\rpoint)$, that is, a section of the projectivized tangent bundle $\mathbb P(T\widetilde{\surf})\to \widetilde{\surf} $. The line field is tangent along each arc of $\widetilde{D}$ and is defined up to homotopy in each polygon cut out by $\widetilde{D}$ by the following foliation:

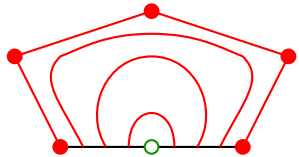
\[\scalebox{1}{
\begin{tikzpicture}[>=stealth,scale=0.6]
\draw[thick, red] (0,0)--(-1,2)--(2,3)--(5,2)--(4,0);
\draw[thick](4,0)--(0,0);
\draw[red,thick,fill=red] (0,0) circle (0.15);
\draw[red,thick,fill=red] (-1,2) circle (0.15);
\draw[red,thick,fill=red] (2,3) circle (0.15);
\draw[red,thick,fill=red] (5,2) circle (0.15);
\draw[red,thick,fill=red] (4,0) circle (0.15);
\draw[dark-green,thick,fill=white] (2,0) circle (0.15);

\draw[thick,red](0.5,0)..controls (0,1) and (-0.5,1.5)..(0,2).. controls (0.5,2.2) and (1,2.5)..(2,2.5)..controls (3,2.5) and   (3.5,2.2)..(4,2)..controls (4.5,1.5) and  (4,1)..(3.5,0);

\draw[thick,red] (1,0)..controls (0.5,0.8) and (1,2)..(2,2).. controls (3,2) and (3.5,0.8)..(3,0);

\draw[thick, red] (1.5,0)..controls (1.5,1) and (2.5,1)..(2.5,0);

\end{tikzpicture}}\]
Note that since $\widetilde{D}$ is $\sigma$-invariant, we may assume that the line field $\widetilde{\eta}$ is $\sigma$-invariant.

Given a non contractible regular closed curve $\gamma$ on $\widetilde{S}\setminus \widetilde{P}_{\rpoint}$ that does not contain any contractible loop, it is possible to define its \emph{winding number} $w_{\widetilde{\eta}}(\gamma)$ with respect to $\widetilde{\eta}$, that is the (algebraic) intersection number between $\dot{\gamma}:\mathbb S^1\to \mathbb P(T\widetilde{\surf})$ and $\widetilde{\eta}\circ\gamma :\mathbb S^1\to \mathbb P(T\widetilde{\surf})$. This winding number is well defined on regular homotopy classes of curves that do not contain a contractible loop. Similarly, if $\gamma:[0,1]\to \widetilde{\surf}\setminus \widetilde{P}_{\rpoint}$ is a regular curve that does not contain any contractible loop such that $\dot{\gamma}(0)$ and $\widetilde{\eta}(\gamma(0))$ are transverse, and $\dot{\gamma}(1)$ and $\widetilde{\eta}(\gamma(1))$ are transverse too, we can define the winding number $w_{\widetilde{\eta}}(\gamma)$ of $\gamma$ with respect to $\widetilde{\eta}$. We refer to \cite{AmiotPlamondonSchroll} or \cite{Opper} for more precise definitions. 

\medskip

We can then define the set $\pi_1^{\rm gr}(\widetilde{\surf},\widetilde{M}_{\gpoint}, \widetilde{P}_{\rpoint})$ of homotopy classes of graded curves. 

\begin{definition}
Let $\gamma$ be a non contractible curve on $\widetilde{\surf}\setminus \widetilde{P}_{\rpoint}$ as in definition \ref{def::pi_1}, or a non contractible regular closed curve $\mathbb S^1\to \widetilde{\surf}\setminus \widetilde{P}_{\rpoint}$. Assume  moreover that $\gamma$ does not contain any contractible loops and that $\gamma$  intersects transversally and minimally the dissection $\widetilde{D}$. A \emph{grading} on $\gamma$ is a map $\mathbf{n}:\gamma(0,1)\cap \widetilde{D}\to \mathbb Z$ satisfying:
$$\mathbf{n}(\gamma(t_{i+1}))=\mathbf{n}(\gamma(t_i))+w_{\eta}(\gamma_{|_{[t_i,t_{i+1}]}}),$$
if $\gamma(t_i)$ and $\gamma(t_{i+1})$ are two consecutive intersections of $\gamma$ with $\widetilde{D}$. More concretely, on $[t_i,t_{i+1}]$, the curve $\gamma$ intersects one polygon cut out by $\widetilde{D}$, and we have 
$$\mathbf{n}(\gamma(t_{i+1}))=\mathbf{n}(\gamma(t_i))+ 1$$ if the boundary segment of  the polygon is on the left of the curve $\gamma_{|_{[t_i,t_{i+1}]}}$, and 
$$\mathbf{n}(\gamma(t_{i+1}))=\mathbf{n}(\gamma(t_i))- 1$$
if the boundary segment lies on the right.

If $(\gamma,\grading)$ and $(\gamma',\grading')$ are two graded curves such that $\gamma$ is homotopic (in the sense of Definition \ref{def::pi_1}) to $\gamma'$ in $\widetilde{\surf}\setminus\widetilde{P}_{\rpoint}$, and such that $\grading(\gamma (t_1))=\grading'(\gamma' (t'_1))$, then their grading coincide, in the sense that  for any $i$ we have 

$$ \gamma(t_i),\gamma(t'_i)\textrm{ lie on the same arc of }\widetilde{D} \textrm{ and } \grading(\gamma(t_i))=\grading'(\gamma'(t'_i)).$$

We denote by $\pi_1^{\rm gr}(\widetilde{\surf},\widetilde{M}_{\gpoint}, \widetilde{P}_{\rpoint})$ the set of graded curves up to  homotopy as defined in Definition \ref{def::pi_1}.

Let $\gamma:\mathbb S^1\to \widetilde{\surf}$ be a non contractible closed curve on $\widetilde{\surf}$ that intersects transversally and minimally the dissection $\widetilde{D}$. One easily sees that it admits a grading if and only if its winding number with respect to the line field $\widetilde{\eta}$ is $0$. We denote by $[\gamma]$ its conjugacy class in the fundamental group of $\widetilde{S}$, that is the free homotopy class of $\gamma$.  We denote by $\pi_1^{\rm gr,free}(\widetilde{\surf})$ the set of non contractible graded closed curves, up to free homotopy.

\end{definition}

Since we have 
\begin{eqnarray*}\grading (\gamma ^{-1}(1-t_i)) & = & \grading (\gamma(t_i))\\
 &= & \grading(\gamma(t_{i+1}))-w_{\eta}(\gamma_{|_{[t_i,t_{i+1}]}})\\
  & = & \grading (\gamma^{-1}(1-t_{i+1))})+w_{\eta}(\gamma^{-1}_{|_{[1-t_{i+1},1-t_{i}]}})
  \end{eqnarray*}
 the equivalence relation $\gamma\sim \gamma^{-1}$ (resp. $[\gamma]\sim [\gamma^{-1}]$) on $\pi_1(\widetilde{\surf},\widetilde{M}_{\gpoint}, \widetilde{P}_{\rpoint})$ (resp. on $\pi_1^{\rm free}(\widetilde{S})$) induces an equivalence relation 
 $(\gamma,\grading)\sim (\gamma^{-1},\grading)$ (resp. $([\gamma],\grading)\sim ([\gamma^{-1}],\grading)$)
 on $\pi_1^{\rm gr}(\surf,\widetilde{M}_{\gpoint},\widetilde{P}_{\rpoint})$ (resp. on $\pi_1^{\rm gr,free}(\widetilde{S})$).

One of the main result in \cite{OpperPlamondonSchroll} is the following

\begin{theorem}\label{thm::OPS}\cite{OpperPlamondonSchroll}
Let $k$ be an algebraically closed field. Let $\widetilde{A}$ be a gentle algebra and $(\widetilde{\surf}, \widetilde{M}_\rpoint,\widetilde{P}_\rpoint,\widetilde{D})$ the associated dissected surface. Then there is a bijection between indecomposable objects of $\cD^{\rm b}(\widetilde{A})$ and the following sets
\begin{enumerate}
\item $\pi_1^{\rm gr}(\widetilde{\surf},\widetilde{M}_{\gpoint}, \widetilde{P}_{\rpoint})/\sim$ where $(\gamma,\grading)\sim (\gamma^{-1},\grading)$;
\item $\pi_1^{\rm gr,free}(\widetilde{\surf})\times k^*/\sim $, where  $([\gamma],\mathbf{n},\lambda)\sim ([\gamma^{-1}],\mathbf{n},\lambda^{-1})$.
\end{enumerate}
\end{theorem}

Note that in this theorem the graded curves wrapping around points in $\widetilde{P}_\rpoint$ correspond to objects that are not in $\mathcal{K}^{\rm b}(\proj \widetilde{A})$.

\subsection{Indecomposable objects in terms of graded curves on $\widetilde{\surf}$}

In this section, the aim is to describe indecomposable objects of $\cD^{\rm b}(A)$ in terms of curves on $\widetilde{\surf}$. For that purpose, the first thing to understand is the action of $\sigma$ on the indecomposable objects of $\cD^{\rm b}(\widetilde{A})$ through the bijection of Theorem \ref{thm::OPS}, and then to apply Corollary \ref{cor::objectsAobjectsAtilde}.

We first fix a piece of notation. For $(\gamma,\mathbf{n})\in \pi_1^{\rm gr}(\surf,M_{\gpoint},\widetilde{P}_{\rpoint})$, we denote by $P_{(\gamma,\mathbf n)}$ the corresponding string object in $\cD^{\rm b}(\widetilde{A})$ of Theorem \ref{thm::OPS}. For $([\gamma],\mathbf n,\lambda)\in \pi_1^{\rm free,gr}(\surf)\times k^*$ we denote by $B_{([\gamma],\mathbf n,\lambda)}$ the corresponding band object in $\cD^{\rm b}(\widetilde{A})$.
 
\begin{lemma}
\begin{enumerate}
\item 
For $(\gamma,\mathbf{n})\in \pi_1^{\rm gr}(\widetilde{\surf},\widetilde{M}_{\gpoint}, \widetilde{P}_{\rpoint})$, we have $(P_{(\gamma,\mathbf n)})^\sigma\simeq P_{(\sigma\circ\gamma,\mathbf n\circ \sigma)}$.
\item For $([\gamma],\mathbf n,\lambda)\in \pi_1^{\rm free,gr}(\widetilde{\surf})\times k^*$, we have $(B_{([\gamma],\mathbf n,\lambda)})^\sigma\simeq B_{([\sigma\circ \gamma],\mathbf n\circ\sigma,\lambda)}$.
\end{enumerate}

\end{lemma}

\begin{proof}
The first statement is proved in Lemma 5.4 in \cite{AmiotBrustle}. Note that the case where $(\gamma,\grading)$ is an infinite arc is similar to the case where $\gamma$ has its endpoints in $\widetilde{M}_{\gpoint}$. 

For the second statement, let us define explicitely the bijection $B$. Let $([\gamma],\mathbf n,\lambda)$ be in $\pi_1^{\rm free,gr}(\surf)\times k^*$. Assume first that $\gamma$ is primitive. Let $\gamma:[0,1]\to \widetilde{\surf}$ be a  regular representative of $[\gamma]$ intersecting transversally the arcs of $\widetilde{D}$. Denote by $0\leq t_0<\cdots<t_\ell<1$ the elements in $[0,1]$ such that $\gamma(t_j)$ belongs to $\widetilde{D}$, and denote by $i_j$ the arc of $\widetilde{D}$ containing $\gamma(t_j)$.

Since $[\gamma]$ is defined up to free homotopy and since $w_{\widetilde{\eta}}(\gamma)=0$, we can assume that $t_0=0$,  that $w_{\widetilde{\eta}}(\gamma_{(0,t_1)})=+1$ and that $w_{\widetilde{\eta}}(\gamma_{(t_\ell,1)})=-1$.  

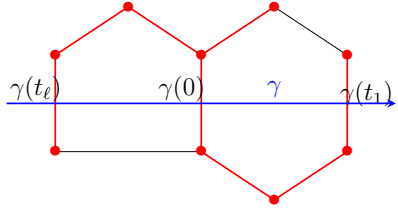
\[\scalebox{0.8}{
\begin{tikzpicture}[>=stealth,scale=0.8]
\draw[thick, red] (0,0)--(0,2)--(1.5,3)--(3,2);
\draw[thick, red] (4.5,3)--(3,2)--(3,0)--(4.5,-1)--(6,0)--(6,2);
\draw (0,0)--(3,0);
\draw (4.5,3)--(6,2);

\node at (0,0) {$\rpoint$};
\node at (0,2) {$\rpoint$};
\node at (1.5,3) {$\rpoint$};
\node at (3,2) {$\rpoint$};
\node at (4.5,3) {$\rpoint$};
\node at (3,0) {$\rpoint$};
\node at (4.5,-1) {$\rpoint$};
\node at (6,0) {$\rpoint$};
\node at (6,2) {$\rpoint$};

\node at (-0.4,1.3) {$\gamma(t_\ell)$};
\node at (2.6,1.3) {$\gamma(0)$};
\node at (6.4,1.2) {$\gamma(t_1)$};

\draw[thick, blue,->] (-1,1)--(7,1);
\node[blue] at (4.5,1.3) {$\gamma$};

\end{tikzpicture}}\]

For $j=0\ldots,\ell$, one can associate a path $p_j(\gamma)$ of the quiver $Q$ as in the following picture (where indices are taken modulo $\ell$).

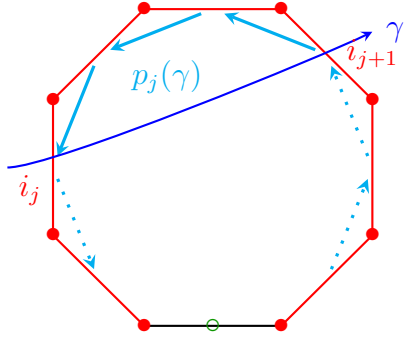
\[\scalebox{1}{
\begin{tikzpicture}[>=stealth,scale=0.6]

\draw[red, thick] (2,-2)--node (A1){} (0,0)--node (A2){}(0,3)--node (A3){}(2,5)--node (A4){}(5,5)--node (A5){}(7,3)--node (A6){}(7,0)--node (A7){}(5,-2);
\draw[thick] (2,-2)--(5,-2);
\node at (3.5,-2) {$\gpoint$};
\node at (0,0) {$\rpoint$};
\node at (0,3) {$\rpoint$};
\node at (2,5) {$\rpoint$};
\node at (5,5) {$\rpoint$};
\node at (7,3) {$\rpoint$};
\node at (7,0) {$\rpoint$};
\node at (5,-2) {$\rpoint$};
\node at (2,-2) {$\rpoint$};

\draw[->,very thick, cyan,  loosely dotted] (A7)--(A6);
\draw[->,very thick, cyan,  loosely dotted] (A6)--(A5);
\draw[->,very thick, cyan] (A5)--(A4);
\draw[->,very thick, cyan] (A4)--(A3);
\draw[->,very thick, cyan] (A3)--(A2);
\draw[->,very thick, cyan, loosely dotted] (A2)--(A1);

\draw[thick, blue,->] (-1,1.5)..controls (0,1.5) and (6,4)..(7,4.5);

\node[red] at (-0.5,1) {$i_j$};
\node[red] at (7,4) {$i_{j+1}$};

\node[cyan] at (2.5,3.5) {$p_j(\gamma)$};
\node[blue] at (7.5,4.5) {$\gamma$};

\end{tikzpicture}}\]

As a graded projective $\widetilde{A}$-module, $B_{([\gamma],\mathbf n,\lambda)}$ is defined to be $$B_{([\gamma],\mathbf n,\lambda}):=\bigoplus_{j=0}^\ell e_{i_j} \widetilde{A}[-{\mathbf n}(\gamma(t_j)].$$
The differential is given by the following $(\ell+1)\times (\ell+1)$ matrix $(d_{(j,k)})_{j,k}$
\begin{itemize} 
\item if $ w_{\eta}(\gamma_{|_{(t_j,t_{j+1})}})=+1$, then  $d_{(j+1,j)}=p_j(\gamma)[-\grading(\gamma(t_j))]$ 

\item if $w_{\eta}(\gamma_{|_{(t_j,t_{j+1})}})=-1$, then  $d_{(j,j+1)}=p_j(\gamma)[-\grading(\gamma(t_{j+1}))]$

\item $d_{(0,\ell)}=\lambda p_{\ell}(\gamma)[-\grading (\gamma (t_{\ell}))]$,

\item all other values of  $d_{(j,k)}$ are $0$. 
\end{itemize}
Note that in case $\ell=1$, then we obtain $d_{(0,1)}=p_0[-\grading (\gamma (t_0))]+\lambda p_{1}[-\grading (\gamma (t_{1})]$.

With the hypothesis on $\gamma$, we define an element $\alpha\in \pi_1(\widetilde{\surf},\widetilde{M}_{\gpoint})$ as in the following picture,
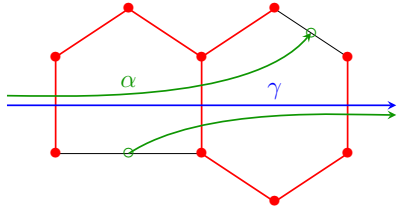
\[\scalebox{0.8}{
\begin{tikzpicture}[>=stealth,scale=0.8]
\draw[thick, red] (0,0)--(0,2)--(1.5,3)--(3,2);
\draw[thick, red] (4.5,3)--(3,2)--(3,0)--(4.5,-1)--(6,0)--(6,2);
\draw (0,0)--(3,0);
\draw (4.5,3)--(6,2);

\node at (0,0) {$\rpoint$};
\node at (0,2) {$\rpoint$};
\node at (1.5,3) {$\rpoint$};
\node at (3,2) {$\rpoint$};
\node at (4.5,3) {$\rpoint$};
\node at (3,0) {$\rpoint$};
\node at (4.5,-1) {$\rpoint$};
\node at (6,0) {$\rpoint$};
\node at (6,2) {$\rpoint$};
\node at (1.5,0) {$\gpoint$};
\node at (5.25,2.5) {$\gpoint$};

\draw[thick, blue,->] (-1,1)--(7,1);
\node[blue] at (4.5,1.3) {$\gamma$};

\draw[thick, dark-green,->] (1.5,0)..controls (3,1) and (6,0.8)..(7,0.8);
\draw[thick, dark-green, ->] (-1,1.2).. controls (0,1.2) and (4,1)..(5.25,2.5);
\node[dark-green] at (1.5,1.5) {$\alpha$};

\end{tikzpicture}}\]
and define a grading on it such that $\grading(\mathbf (\alpha(t_j)))=\grading(\mathbf (\gamma(t_j)))$ for $j=1,\ldots,\ell$. Then one immediately sees that the map $(1,\lambda):e_{i_0}\widetilde{A}^2\to e_{i_0}\widetilde{A}$ induces a triangle 
$$\xymatrix{P_{(\alpha,\grading)}\ar[r] & B_{([\gamma],\grading,\lambda)}\ar[r] & e_{i_0}\widetilde{A}[-\grading (\gamma(0))]\ar[r] & P_{(\alpha,\grading)}[1]}.$$

Then we obtain statement $(2)$ for primitive curves using statement $(1)$.

Finally if $\gamma$  is a closed curve which is not primitive, it is a power of a primitive curve $\alpha$.  Then the object $B_{([\gamma],\grading,\lambda)}$ is an iterated extension of the object $B_{([\alpha],\grading,\lambda)}$. Hence an easy induction gives statement $(2)$. 

\end{proof}

We have $P_{(\gamma,\grading)}\simeq P_{(\gamma',\grading')}$ if and only if $\gamma=\gamma'$ or $\gamma=\gamma'^{-1}$ and $\grading=\grading'$.  We have 
$B_{([\gamma],\grading,\lambda)}\simeq B_{([\gamma'],\grading',\lambda')}$ if and only if $([\gamma],\grading,\lambda)=([\gamma'],\grading',\lambda')$ or $([\gamma],\grading,\lambda)=([\gamma'^{-1}],\grading',\lambda^{-1})$. Hence the indecomposable objects of $\cD^{\rm b}(A)$ are in bijection with the following five sets:

\begin{enumerate}
\item[$(\widetilde{\surf}1)$] $\left\{ (\gamma,\grading)\in \pi_1^{\rm gr}(\widetilde{\surf},\widetilde{M}_{\gpoint},\widetilde{P}_{\rpoint})\ |\  \gamma^{-1}\neq \sigma \gamma\right\}/\sim$ 

where $(\gamma,\grading)\sim (\sigma\gamma,\grading\circ\sigma)\sim (\gamma^{-1},\grading)$,

\item[$(\widetilde{\surf}2)$] $\left\{ (\gamma,\grading,\epsilon)\in \pi_1^{\rm gr}(\widetilde{\surf},\widetilde{M}_{\gpoint}, \widetilde{P}_{\rpoint})\times \{\pm 1\}, \ |\  \sigma\gamma=\gamma^{-1}\right\}/\sim$ 

where $(\gamma,\grading,\epsilon)\sim (\gamma^{-1},\grading,\epsilon)$.
\item[$(\widetilde{\surf}3)$] $\left\{([\gamma],\grading, \lambda)\in \pi_1^{\rm free,gr}(\widetilde{\surf})\times k^*\ |\ [\sigma\gamma]\neq [\gamma],[\gamma^{-1}]\right\}/ \sim$ 

where $([\gamma],\grading, \lambda)\sim ([\gamma^{-1}],\grading, \lambda^{-1})\sim ([\sigma\gamma],\grading\circ\sigma, \lambda)$;

\item[$(\widetilde{\surf}3')$] $\left\{([\gamma],\grading, \lambda,\epsilon)\in \pi_1^{\rm free,gr}(\widetilde{\surf})\times k^*\times\{\pm 1\}\ |\ [\sigma\gamma]=[\gamma]\right\}/ \sim$ 

where $([\gamma],\grading, \lambda,\epsilon)\sim ([\gamma^{-1}],\grading, \lambda^{-1},\epsilon)$;

\item[$(\widetilde{\surf}4)$] $\left\{([\gamma],\grading, \lambda)\in \pi_1^{\rm free,gr}(\widetilde{\surf})\times k^*\setminus \{\pm 1\}\ |\ [\sigma\gamma]=[\gamma^{-1}]\right\}/ \sim$ 

where $([\gamma],\grading, \lambda)\sim ([\gamma^{-1}],\grading, \lambda^{-1})\sim ([\sigma\gamma],\grading\circ\sigma, \lambda)$;

\item[$(\widetilde{\surf}5)$] $\left\{([\gamma],\grading, \lambda,\epsilon)\in \pi_1^{\rm free,gr}(\widetilde{\surf})\times \{\pm 1\}\times\{\pm 1\}\ |\ [\sigma\gamma]=[\gamma^{-1}]\right\}/ \sim$ 

where $([\gamma],\grading, \lambda,\epsilon)\sim ([\gamma^{-1}],\grading, \lambda^{-1},\epsilon)$

\end{enumerate}

\section{Indecomposables in term of graded curves  on the orbifold}
 The aim of this section is to use the projection map $p:\widetilde{\surf}\to \surf$ to describe the sets $(\widetilde{\surf}_1)-(\widetilde{\surf}_5)$ in terms of graded curves on the orbifold $\surf$.

  We denote by $\pi_1^{\rm orb}(\surf, M_{\gpoint}, P_{\rpoint})$  the quotient
of $\pi_1(\surf \setminus X_\cross, M_{\gpoint},P_{\rpoint})$ by the equivalence relation given by

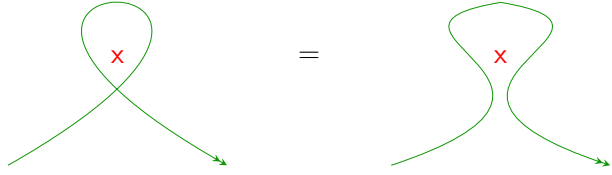
\begin{center} 
 \scalebox{0.9}{
 
\begin{tikzpicture}[>=stealth,scale=0.8]

 \node (P) at (0 , 0) {$\cross$};

 \draw[->>,dark-green] (-2,-2) .. controls (5,2) and (-5,2) .. (2,-2);
 
 \node at (3.5,0) {$=$};
 
\begin{scope}[xshift=7cm,yshift=0cm]
 
 \node at (0,0) {$\cross$};

 \draw[->>,dark-green] (-2,-2) .. controls (2.5,-0.5) and (-3,0.5) .. (0,1) .. controls (3,0.5) and (-2.5,-0.5) .. (2,-2); 
\end{scope} 

\end{tikzpicture}
}
\end{center}
 The set $\pi_1^{\rm orb,free}(\surf)$ is the set of conjugacy classes of the fundamental orbifold group $\pi_1^{\rm orb}(\surf)$.

Now recall from \cite[Proposition 5.5]{AmiotPlamondon} that there is a map $$\Phi:\pi_1(\widetilde{\surf},\widetilde{M}_{\gpoint}, \widetilde{P}_{\rpoint})\to \pi_1^{\rm orb}(\surf,M_{\gpoint},P_{\rpoint})$$ which is faithful in the sense that any non contractible loop in $\pi_1(\widetilde{\surf},\widetilde{M}_{\gpoint}, \widetilde{P}_{\rpoint})$ gives rise to a non zero element in $\pi_1^{\rm orb}(\surf,M_{\gpoint},P_{\rpoint})$. This map induces a well defined map by \cite[Proposition 5.14]{AmiotPlamondon}
$$\Psi: \pi_1^{\rm free}(\widetilde{\surf})\to \pi_1^{\rm orb, free}(\surf). $$

 These maps are essential to translate the bijection between indecomposable objects with the sets $(\widetilde{\surf}1)$ to $(\widetilde{\surf}5)$ above in terms of graded curves on the orbifold $\surf$. 

\subsection{String objects}\label{section4.1}

The first step consists of the definition of the set $\pi_1^{\rm orb, gr}(\surf,M_{\gpoint}, P_{\rpoint})$ together with a map

$$\pi_1^{\rm gr}(\widetilde{\surf},\widetilde{M}_{\gpoint}, \widetilde{P}_{\rpoint})\to \pi_1^{\rm orb, gr}(\surf,M_{\gpoint}, P_{\rpoint}).$$

Note that since $\widetilde{\eta}$ is $\sigma$-invariant, we obtain a line field $\eta$ on $\surf\setminus (X_\cross\cup P_{\rpoint})$. Hence there is a natural map

$$\pi_1^{\rm gr}(\widetilde{\surf}\setminus \widetilde{X}_{\cross},\widetilde{M}_{\gpoint}, \widetilde{P}_{\rpoint})\to \pi_1^{\rm gr}(\surf\setminus X_{\cross},M_{\gpoint}, P_{\rpoint}).$$

\begin{definition}
 Let $\gamma$ be in $\mathcal{C}^1((0,1), \surf\setminus (X_\cross\cup P_{\rpoint})$ such that its preimages $\widetilde{\gamma}$ and $\sigma\widetilde{\gamma}$ in $\mathcal{C}^1((0,1),\widetilde{\surf}\setminus (\widetilde{X}_\cross\cup \widetilde{P}_{\rpoint})$ do not contain any contractible loops and intersect transversally and minimally the dissection $\widetilde{D}$.
 
Then, one defines a grading on $\gamma$ as a map $\grading: \gamma(0,1)\cap D\to \mathbb Z$ such that 
$$\mathbf{n}(\gamma(t_{i+1}))=\mathbf{n}(\gamma(t_i))+w_{\eta}(\gamma_{|_{[t_i,t_{i+1}]}}),$$
if $\gamma(t_i)$ and $\gamma(t_{i+1})$ are two consecutive intersections of $\gamma$ with $D$.

\end{definition}
Since the map $\Phi:\pi_1(\widetilde{\surf},\widetilde{M}_{\gpoint}, \widetilde{P}_{\rpoint})\to \pi_1^{\rm orb}(\surf,M_{\gpoint}, P_{\rpoint})$ is surjective, any element in $\pi_1^{\rm orb} (\surf,M_{\gpoint}, P_{\rpoint})$ has a representant that can be gradable. 

We would like now to check that the grading is well-defined on the set $\pi_1^{\rm orb}(\surf, M_{\gpoint}, P_{\rpoint}).$ This comes from the following two facts:

\begin{enumerate}
\item If $(\widetilde{\gamma},\widetilde{\grading})$ is a graded curve in $\widetilde{\surf}$, and $(\gamma,\grading)$ is a graded curve in $\surf$ such that $\Phi(\widetilde{\gamma})=\gamma$ and $\grading (\gamma (t_1))=\widetilde{\grading }(\widetilde{\gamma}(t_1))$, then for any $i$, we have 
$\grading (\gamma (t_i))=\widetilde{\grading }(\widetilde{\gamma}(t_i)).$ This comes from the fact that  $\eta$ is the projection of $\widetilde{\eta}$.

\item If $(\gamma,\grading)$ and $(\gamma',\grading')$ be two graded curves on $\surf$ that have the same grading at their first intersection point with $D$, then they admit the same grading on any intersection point with $D$. Indeed their preimages starting at the same point are homotopic, so they admit the same grading in $\widetilde{\surf}$ with respect to the dissection $\widetilde{D}$.
\end{enumerate}

We denote by $\pi_1^{\rm orb, gr}(\surf,M_{\gpoint},P_{\rpoint} )$, the set of graded curves up to homotopy, which is now well-defined. It comes then with a natural surjective map 
$$\pi_1^{\rm orb,gr}(\surf,M_{\gpoint},P_{\rpoint})\to \pi_1^{\rm orb}(\surf,M_{\gpoint}, P_{\rpoint})$$ whose fiber is in bijection with $\mathbb Z$.

An element $\widetilde{\gamma}$ in $\pi_1(\widetilde{\surf},\widetilde{M}_{\gpoint},\widetilde{P}_{\rpoint})$ satisfying $\sigma \widetilde{\gamma}=\widetilde{\gamma}^{-1}$ has its image $\gamma$ in $\pi_1^{\rm orb}(\surf, M_{\gpoint},P_{\rpoint})$ that satisfies $\gamma=\gamma^{-1}$. Conversely, an element $\gamma$ in $\pi_1^{\rm orb}(\surf, M_{\gpoint},P_{\rpoint})$ such that $\gamma=\gamma^{-1}$ has its preimages $\widetilde{\gamma}$ and $\sigma \widetilde{\gamma}$ satisfying $\sigma \widetilde{\gamma}.\widetilde{\gamma}=1$ since there is no torsion in $\pi_1(\widetilde{\surf}, \widetilde{M}_{\gpoint}, \widetilde{P}_{\rpoint})$.

Therefore the sets $(\widetilde{S}1)$ and $(\widetilde{\surf}2)$ described above are respectively in bijection with
\begin{enumerate}

\item[$(\surf 1)$] $\left\{(\gamma,\grading)\in \pi_1^{\rm orb,gr}(\surf,M_{\gpoint}, P_{\rpoint})\ |\ \gamma\neq \gamma^{-1}\right\}/\sim$ where $(\gamma,\grading)\sim (\gamma^{-1},\grading)$;
\item[$(\surf 2)$] $\left\{ (\gamma,\grading,\epsilon)\in \pi_1^{\rm orb,gr}(\surf,M_{\gpoint}, P_{\rpoint})\times \{\pm 1\}\ |\ \gamma=\gamma^{-1}\right\}$
\end{enumerate}

\subsection{Band objects}\label{section4.2}

Here again, we first define  the notion of gradable closed curves on the orbifold $\surf$.

Let $[\gamma] \in \pi_1^{\rm orb, free} (\surf)$ be represented by a smooth curve $\gamma$ without contractible loops and intersecting transversally $D$. Denote by $x_0=\gamma (0)$ its starting point, and by $x_0^+$, and $x_0^-$ its preimages in $\widetilde{\surf}$. There exists a curve $\widetilde{\gamma}\in \mathcal{C}^1((0,1),\surf)$ satisfying :

\begin{itemize}
\item $\widetilde{\gamma}$ does not contain any contractible loops;
\item $\widetilde{\gamma}(0)=x_0^+$, and $\widetilde{\gamma}(1)\in \{x_0^+,x_0^-\}$;
\item
$\dot{\widetilde{\gamma}} (0)=\dot{\widetilde{\gamma}}(1)$ if $\widetilde{\gamma}(1)=x_0^+$;
\item $\sigma(\dot{\widetilde{\gamma}}(0))=(T\sigma)\dot{\widetilde{\gamma}}(1)$ if $\widetilde{\gamma}(1)=x_0^-$, and where $T\sigma: T_{x_0^-}\widetilde{\surf} \to T_{x_0^+}\widetilde{\surf}$ is induced by the diffeomorphism $\sigma$.
\end{itemize}
Then the winding number  of $\widetilde{\gamma}$ with respect to $\widetilde{\eta}$ is defined, and so is the winding number of $\gamma=p\widetilde{\gamma}$ with respect to $\eta$.
Moreover we have 
$$w_{\bar{\eta}} (\gamma)=w_{\eta}(\widetilde{\gamma}).$$

Furthermore since the preimage $\tilde{\gamma}$ (starting in $x_0^+$) is unique up to homotopy, the winding number of $[\gamma]$ is well defined as a map 
$$w_{\eta}:\pi_1^{\rm orb, free}(\surf)\longrightarrow \mathbb Z.$$

\begin{definition}
Denote by $\mathbb S^1$ the segment $[0,1]$, where $0$, and $1$ are identified.
 Let $\gamma:\mathbb S^1 \to\surf\setminus (X_\cross\cup P_{\rpoint})$ be a closed smooth map with $\gamma(0)=x_0$ and such that its preimage $\widetilde{\gamma}:[0,1]\to \widetilde{\surf}$ on $\widetilde{\surf}$ starting at $x_0^+$ is as above. 
 A \emph{grading} on $\gamma$ is a map $\mathbf{n}:\gamma(\mathbb S^1)\cap D\to \mathbb Z$ satisfying:
$$\mathbf{n}(\gamma(t_{i+1}))=\mathbf{n}(\gamma(t_i))+w_{\eta}(\gamma_{|_{[t_i,t_{i+1}]}}),$$
if $\gamma(t_i)$ and $\gamma(t_{i+1})$ are two consecutive intersections of $\gamma$ with $D$.
\end{definition}

Note that if $\gamma$ has a grading and if $\widetilde{\gamma}$ is not closed (that is $\widetilde{\gamma}(1)=x_0^-$), then one can consider the closed curve $\beta:=\sigma\widetilde{\gamma}.\widetilde{\gamma}:[0,2]\to \widetilde{\surf}$. The grading $\grading$ defines a grading $\widetilde{\grading}$ on $\beta$ (and on $\widetilde{\gamma}$) with for any $i$
$\widetilde{\mathbf{n}}(\widetilde{\gamma}(t_{i+1}))=\widetilde{\mathbf{n}}(\widetilde{\gamma}(t_i))+w_{\widetilde{\eta}}(\widetilde{\gamma}_{|_{[t_i,t_{i+1}]}})$ and such that

\begin{eqnarray*} (1) \label{gradingeq} \quad\widetilde{\mathbf{n}}((\widetilde{\gamma})(t_\ell)) & = & \widetilde{\grading}(\beta(t_\ell))\\ & = & \widetilde{\mathbf{n}}(\beta(1+t_1))-w_{\widetilde{\eta}}(\beta_{|_{[t_\ell,1+t_1]}})\\ & =& \widetilde{\mathbf{n}}(\sigma(\widetilde{\gamma})(t_1)))-w_{\widetilde{\eta}}(\beta{|_{[t_\ell,1+t_1]}})\\ & = & \widetilde{\mathbf{n}}(\widetilde{\gamma}(t_1)))-w_{\widetilde{\eta}}(\beta{|_{[t_\ell,1+t_1]}}) 
\end{eqnarray*}

Then, with the same argument as before, we see that if the gradings of two graded closed curves that are equal when viewed in $\pi_1^{\rm orb, free}(\surf)$ coincide at their first point, then they coincide at each intersection point with the dissection. Therefore, the set $\pi_1^{\rm orb, free, gr}(\surf)$ is well-defined. 

Moreover we have the following

\begin{proposition}\label{prop gradability and winding}
Let $[\gamma]\in \pi_1^{\rm orb, free}(\surf)$. Then $[\gamma]$ is gradable if and only if $w_{\eta}([\gamma])=0$.
\end{proposition}

\begin{proof}
Let $\gamma$ representing $[\gamma]$ be such that its pre-image $\widetilde{\gamma}$ is as before. If $\widetilde{\gamma}$ is closed on $\surf$, this is clear since we have 
$$\gamma \textrm{ gradable }\Leftrightarrow\ \widetilde{\gamma} \textrm{ gradable }\Leftrightarrow w_{\widetilde{\eta}}(\widetilde{\gamma})=0\ \Leftrightarrow\ w_{\eta}(\gamma)=0.$$

If $\widetilde{\gamma}$ is not closed, then we have 
\begin{eqnarray*}\gamma \textrm{ gradable } & \Leftrightarrow &  \widetilde{\gamma} \textrm{ is gradable with condition \eqref{gradingeq}}\\ &  \Leftrightarrow & w_{\widetilde{\eta}}(\sigma\widetilde{\gamma}.\widetilde{\gamma})=0\ \\ & \Leftrightarrow&\ w_{\eta}(\gamma)=0,\end{eqnarray*}
since $w_{\widetilde{\eta}}(\sigma\widetilde{\gamma}.\widetilde{\gamma})=2w_{\widetilde{\eta}}(\widetilde{\gamma})=2w_{\eta}(\gamma).$

\end{proof}

Therefore we obtain a map 
$$\pi_{1}^{\rm orb, free, gr}(\surf)\longrightarrow \pi_1^{\rm orb, free}(\surf),$$ whose image consists of curves with winding number $0$, and whose fiber is in bijection with $\mathbb Z$.

\begin{definition}
We call an element $\gamma\in\pi_1(\surf,x_0)$ \emph{primitive} if it is torsionfree, and if it is a generator of the maximal cyclic group containing it.

Hence if $\gamma\in \pi_1^{\rm orb}(\surf,x_0)$ satisfies $\gamma^2\neq 1$ then $\gamma$ is torsionfree, and so can be written in a unique way as a positive power of a primitive element.

\end{definition}

Now, a small adaptation of Corollary 5.18 in \cite{AmiotPlamondon} yields the following.

\begin{proposition}\label{prop::bijection courbes S et Stilde}

Let $\Psi:\pi_1^{\rm free}(\widetilde{\surf})\to \pi_1^{\rm orb,free}(\surf)$ be the map induced by the projection $p:\widetilde{\surf}\to \surf$. 
 \begin{enumerate}
  \item We have a bijection between the following sets:
    \begin{enumerate}
     \item 
     \( \Big\{ \{[\widetilde{\gamma}], [\sigma\widetilde{\gamma}] \} \ | \ [\widetilde{\gamma}] \in \pi_1^{\rm free}(\widetilde{\surf}) 
              \textrm{ primitive with $w_{\widetilde{\eta}}(\widetilde{\gamma})=0$  and } [\sigma\widetilde{\gamma}] \neq [\widetilde{\gamma}], [\widetilde{\gamma}^{-1}] \Big\} ;
           \)
     \item \(
             \Big\{ [\gamma] \in \pi_1^{\rm orb,free}(\surf) \ | \ [\gamma]\in {\rm Im} \Psi \ \textrm{primitive with }w_{\eta}(\gamma)=0\  , [\gamma] \neq [\gamma^{-1}] \Big\} .
           \)
    \end{enumerate}
    
  \item We have a bijection between the sets
    \begin{enumerate}
     \item \( \Big\{ \{[\widetilde{\gamma}], [\sigma\widetilde{\gamma}] \} \ | \ [\widetilde{\gamma}] \in \pi_1^{\rm free}(\widetilde{\surf}) 
              \textrm{ primitive with $w_{\widetilde{\eta}}(\widetilde{\gamma})=0$ and such that } [\sigma\widetilde{\gamma}]  = [\widetilde{\gamma}^{-1}] \Big\} ;
           \)
     \item \(
             \Big\{ [\gamma] \in \pi_1^{\rm orb,free}(\surf) \ | \ [\gamma] \textrm{  primitive with $w_{\eta}(\gamma)=0$ and } [\gamma] = [\gamma^{-1}] \Big\} .
           \)
    \end{enumerate}
    
  \item We have a bijection between the sets
    \begin{enumerate}
     \item \( \Big\{ [\widetilde{\gamma}] \ | \ [\widetilde{\gamma}] \in \pi_1^{\rm free}(\widetilde{\surf}) 
              \textrm{ primitive with $w_{\widetilde{\eta}}(\widetilde{\gamma})=0$ and such that } [\sigma\widetilde{\gamma}]  = [\widetilde{\gamma}] \Big\} \times k^* \times \bZ_2;
           \)
     \item \(
             \Big\{ [\alpha] \in \pi_1^{\rm orb,free}(\surf) \ | \ [\alpha] \notin {\rm Im} \Psi \textrm{ primitive with $w_{\eta}(\alpha)=0$} \Big\} \times k^*.
           \)
    \end{enumerate}

 \end{enumerate}
\end{proposition}

\begin{proof}
The bijections in items $(1)$ and $(2)$ are induced by $\Psi$, and we always have $w_{\widetilde{\eta}}([\widetilde{\gamma}])=w_{\eta} (\Psi[\widetilde{\gamma}])$. Hence the proof here follows from $(1)$ and $(2)$ of Corollary 5.18 in \cite{AmiotPlamondon}.

Bijection $(3)$  is constructed as follows (see proof of Corollary 5.18 in \cite{AmiotPlamondon}) :  for any $[\widetilde{\gamma}]\in\pi_1^{\rm free}(\widetilde{\surf})$ primitive such that $[\sigma\widetilde{\gamma}]=[\widetilde{\gamma}]$  there exists a primitive element $[\alpha]\in \pi_1^{\rm orb, free}(\surf)$ such that $\Psi ([\widetilde{\gamma}])=[\alpha^2].$ If $w_{\widetilde{\eta}}(\widetilde{\gamma})=0$, then $w_{\eta}(\alpha^2)=0=2w_{\eta}(\alpha)$. Thus $\alpha$ has winding number zero if and only if so does $\widetilde{\gamma}$.  We associate to $([\widetilde{\gamma}],\lambda,\pm 1)$ the element $([\alpha],\pm\lambda')$ where $\lambda'$ is a square root of $\lambda$ in $k$.

\end{proof}

Then combining Propositions  \ref{prop gradability and winding} and \ref{prop::bijection courbes S et Stilde}, one easily gets that the sets $(\widetilde{\surf}3)$ and $(\widetilde{\surf}3')$ (resp $(\widetilde{\surf}4)$, resp $(\widetilde{\surf}5)$) are in bijection respectively with

\begin{enumerate}

\item [$(\surf 3)$]$\{ ([\gamma],\mathbf{n},\lambda)\in \pi_1^{\rm orb,free, gr}(\surf)\times k^*\ |\ [\gamma]\neq [\gamma^{-1}]\}/\sim$ where $([\gamma],\mathbf{n},\lambda)\sim ([\gamma^{-1}],\mathbf{n},\lambda^{-1})$;
\item[$(\surf 4)$] $\{ ([\gamma],\mathbf{n},\lambda)\in \pi_1^{\rm orb,free, gr}(\surf)\times k^*\setminus\{\pm 1\}\ |\ [\gamma]= [\gamma^{-1}], \gamma^2\neq 0 \}/\sim$;
\item[$(\surf 5)$] $\{ ([\gamma],\mathbf{n},\epsilon,\epsilon')\in \pi_1^{\rm orb,free, gr}(\surf)\times \{\pm 1\}^2\ |\ [\gamma]= [\gamma^{-1}], \gamma^2\neq 0\}.$
\end{enumerate}

\section{Example}

\subsection{The surfaces $\surf$ and $\widetilde{\surf}$, and the algebras $A$ and $\widetilde{A}$}\label{section5.1}

Let $(\surf,M_{\gpoint},P_{\rpoint}, X_\cross)$ be a cylinder with one marked point on each boudanry component, one puncture and one orbifold point. We consider the following $\cross$-dissection $D$, together with its corresponding skew-gentle algebra.

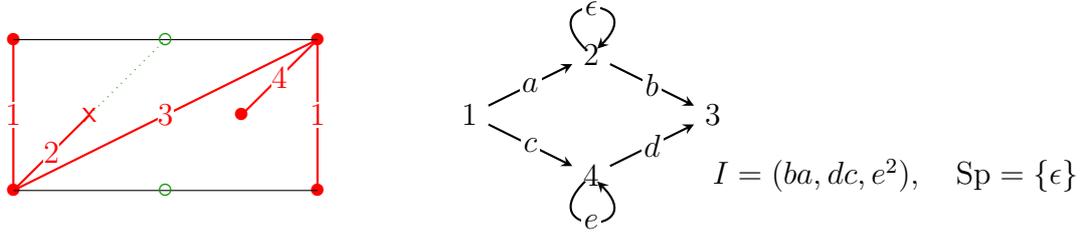
\begin{center}\scalebox{1}{
\begin{tikzpicture}[>=stealth,scale=1]

\node at (0,0) {$\rpoint$};
\node at (0,2) {$\rpoint$};
\node at (4,0) {$\rpoint$};
\node at (3,1) {$\rpoint$};
\node at (4,2) {$\rpoint$};
\node at (1,1) {$\cross$};

\draw[red, thick] (0,2)--node[red,fill=white, inner sep=1pt]{$1$}(0,0)--node[red,fill=white, inner sep=0pt]{$3$}(4,2)--node[red,fill=white, inner sep=1pt]{$1$}(4,0);
\draw[red,thick] (0,0)--node[red, fill=white, inner sep=0pt]{$2$}(1,1);
\draw[red,thick] (4,2)--node[red,fill=white, inner sep=0pt]{$4$}(3,1);
\draw (0,0)--(4,0);
\draw (0,2)--(4,2);

\node at (2,0) {$\gpoint$};
\node at (2,2) {$\gpoint$};

\draw[dark-green,dotted] (1,1)--(2,2);

\begin{scope}[xshift=6cm, yshift=1cm, scale=0.8]
\node (1) at (0,0) {$1$};
\node (2) at (2,1) {$2$};
\node (4) at (2,-1) {$4$};
\node (3) at (4,0) {$3$};

\draw[->, thick] (1)--node[fill=white, inner sep= 1pt]{$a$} (2);
\draw[->, thick] (1)--node[fill=white, inner sep= 1pt]{$c$} (4);
\draw[->, thick] (2)--node[fill=white, inner sep= 1pt]{$b$} (3);
\draw[->, thick] (4)--node[fill=white, inner sep= 1pt]{$d$} (3);

\draw[->, thick] (1.9,1.1).. controls (1,2) and (3,2)..node[fill=white, inner sep= 1pt]{$\epsilon$} (2.1,1.1);

\draw[->, thick] (1.9,-1.1).. controls (1,-2) and (3,-2)..node[fill=white, inner sep= 1pt]{$e$} (2.1,-1.1);

\node at (7,-1) {$I=(ba,dc,e^2), \quad {\rm Sp}=\{\epsilon\}$};

\end{scope}

\end{tikzpicture}}
\end{center}

Note that the quiver $Q$ is not the Gabriel quiver of the algebra $A$. 
But one can easily check that there is an isomorphism with the algebra presented by the following quiver $\bar{Q}$ and the set of admissible relations $\bar{I}$:
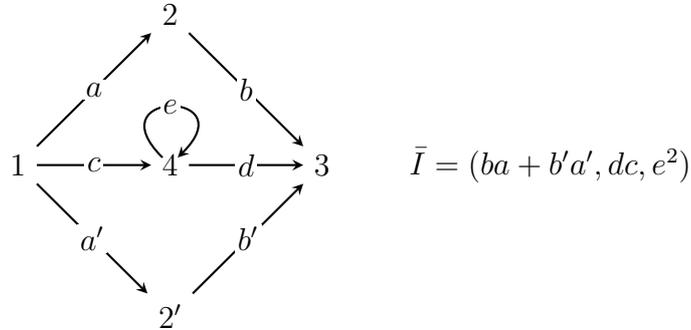
\begin{center}\scalebox{1}{
\begin{tikzpicture}[>=stealth,scale=1]
\begin{scope}[xshift=6cm, yshift=1cm, scale=1]
\node (1) at (0,0) {$1$};
\node (2) at (2,2) {$2$};
\node (2') at (2,-2) {$2'$};
\node (3) at (4,0) {$3$};
\node (4) at (2,0) {$4$};

\draw[->, thick] (1)--node[fill=white, inner sep= 1pt]{$a$} (2);
\draw[->, thick] (1)--node[fill=white, inner sep= 1pt]{$a'$} (2');
\draw[->, thick] (2)--node[fill=white, inner sep= 1pt]{$b$} (3);
\draw[->, thick] (2')--node[fill=white, inner sep= 1pt]{$b'$} (3);
\draw[->, thick] (1)--node[fill=white, inner sep= 1pt]{$c$} (4);
\draw[->, thick] (4)--node[fill=white, inner sep= 1pt]{$d$} (3);

\draw[->, thick] (1.9,0.1).. controls (1,1) and (3,1)..node[fill=white, inner sep= 1pt]{$e$} (2.1,0.1);

\node at (7,0) {$\bar{I}=(ba+b'a',dc,e^2)$};

\end{scope}
\end{tikzpicture}}
\end{center}

The dissected surface $(\widetilde{\surf}, \widetilde{M}_{\gpoint},\widetilde{P}_{\rpoint}, \widetilde{D})$, and the $\bZ_2$-gentle algebra $\widetilde{A}$ associated to the skew-gentle algebra $A$ are as follows.

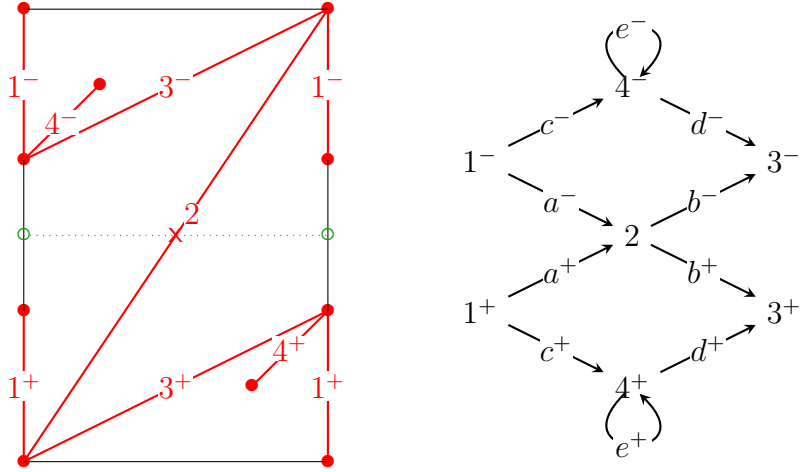
\begin{center}\scalebox{1}{
\begin{tikzpicture}[>=stealth,scale=1]

\node at (0,1) {$\rpoint$};
\node at (0,3) {$\rpoint$};
\node at (4,1) {$\rpoint$};
\node at (1,2) {$\rpoint$};
\node at (4,3) {$\rpoint$};
\node at (2,0) {$\cross$};

\node at (0,-1) {$\rpoint$};
\node at (0,-3) {$\rpoint$};
\node at (4,-1) {$\rpoint$};
\node at (3,-2) {$\rpoint$};
\node at (4,-3) {$\rpoint$};

\draw[red, thick] (0,3)--node[red,fill=white, inner sep=1pt]{$1^-$}(0,1)--node[red,fill=white, inner sep=0pt]{$3^-$}(4,3)--node[red,fill=white, inner sep=1pt]{$1^-$}(4,1);

\draw[red, thick] (0,-1)--node[red,fill=white, inner sep=1pt]{$1^+$}(0,-3)--node[red,fill=white, inner sep=0pt]{$3^+$}(4,-1)--node[red,fill=white, inner sep=1pt]{$1^+$}(4,-3);

\draw[red,thick] (0,1)--node[red, fill=white, inner sep=0pt]{$4^-$}(1,2);
\draw[red,thick] (4,-1)--node[red, fill=white, inner sep=0pt]{$4^+$}(3,-2);

\draw[red,thick] (4,3)--node[red,fill=white, inner sep=0pt, xshift=6pt,yshift=8pt]{$2$}(0,-3);

\draw (0,3)--(4,3);
\draw (0,1)--(0,-1);
\draw (4,1)--(4,-1);
\draw (0,-3)--(4,-3);

\node at (0,0) {$\gpoint$};
\node at (4,0) {$\gpoint$};

\draw[dark-green,dotted] (0,0)--(4,0);

\begin{scope}[xshift=6cm]
\node (1-) at (0,1) {$1^-$};
\node (1+) at (0,-1) {$1^+$};
\node (2) at (2,0) {$2$};
\node (3-) at (4,1) {$3^-$};
\node (3+) at (4,-1) {$3^+$};
\node (4-) at (2,2) {$4^-$};
\node (4+) at (2,-2) {$4^+$};

\draw[thick, ->] (1-)--node[fill=white, inner sep=0pt]{$c^-$}(4-);
\draw[thick, ->] (1+)--node[fill=white, inner sep=0pt]{$c^+$}(4+);
\draw[thick, ->] (1-)--node[fill=white, inner sep=0pt]{$a^-$}(2);
\draw[thick, ->] (1+)--node[fill=white, inner sep=0pt]{$a^+$}(2);
\draw[thick, ->] (2)--node[fill=white, inner sep=0pt]{$b^-$}(3-);
\draw[thick, ->] (2)--node[fill=white, inner sep=0pt]{$b^+$}(3+);
\draw[thick, ->] (4-)--node[fill=white, inner sep=0pt]{$d^-$}(3-);
\draw[thick, ->] (4+)--node[fill=white, inner sep=0pt]{$d^+$}(3+);

\draw[->, thick] (1.9,2.1).. controls (1,3) and (3,3)..node[fill=white, inner sep= 1pt]{$e^-$} (2.1,2.1);

\draw[->, thick] (1.9,-2.1).. controls (1,-3) and (3,-3)..node[fill=white, inner sep= 1pt]{$e^+$} (2.1,-2.1);

\node at (2,-4) {$\widetilde{I}=(b^-a^-, b^+a^+, d^-c^-, d^+c^+, (e^-)^2,(e^+)^2)$ };

\end{scope}

\end{tikzpicture}}
\end{center}

\subsection{Objects in the sets $(\surf 1)$ and $(\surf 2)$}

Let $(\gamma,\grading) \in \pi_1^{{\rm orb,gr}}(\surf,M_{\gpoint}, P_{\rpoint})$ be as follows.

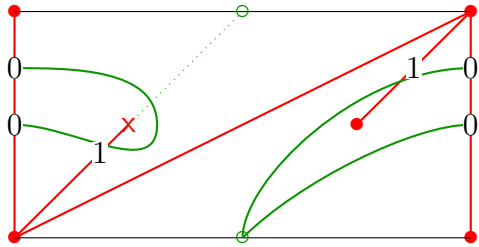
\begin{center}\scalebox{1}{
\begin{tikzpicture}[>=stealth,scale=1.5]

\node at (0,0) {$\rpoint$};
\node at (0,2) {$\rpoint$};
\node at (4,0) {$\rpoint$};
\node at (3,1) {$\rpoint$};
\node at (4,2) {$\rpoint$};
\node at (1,1) {$\cross$};

\draw[red, thick] (0,2)--(0,0)--(4,2)--(4,0);
\draw[red,thick] (0,0)--(1,1);
\draw[red,thick] (4,2)--(3,1);
\draw (0,0)--(4,0);
\draw (0,2)--(4,2);

\node at (2,0) {$\gpoint$};
\node at (2,2) {$\gpoint$};

\draw[dark-green,dotted] (1,1)--(2,2);

\draw[dark-green, thick] (0,1).. controls (0.5,1) and  (1.25,0.5)..(1.25,1).. controls (1.25,1.5) and (0.5,1.5)..(0,1.5);

\draw[dark-green, thick] (2,0).. controls (2.5,0.5) and (3.5,1).. (4,1);
\draw[dark-green, thick] (2,0).. controls (2,0.5) and (3,1.5).. (4,1.5);

\node[fill=white, inner sep=0pt] at (0,1) {$0$};
\node[fill=white, inner sep=0pt] at (0,1.5) {$0$};
\node[fill=white, inner sep=0pt] at (4,1) {$0$};
\node[fill=white, inner sep=0pt] at (4,1.5) {$0$};
\node[fill=white, inner sep=0pt] at (3.5,1.5) {$1$};
\node[fill=white, inner sep=0pt] at (0.75,0.75) {$1$};

\end{tikzpicture}}\end{center}

The element $\gamma$ satisfies $\gamma^2\neq 1$, therefore $(\gamma,\grading)$ is in the set $(\surf 1)$ and has two preimages in $\pi_1^{\rm gr}(\widetilde{\surf},\widetilde{M}_{\gpoint}, \widetilde{P}_{\rpoint}).$

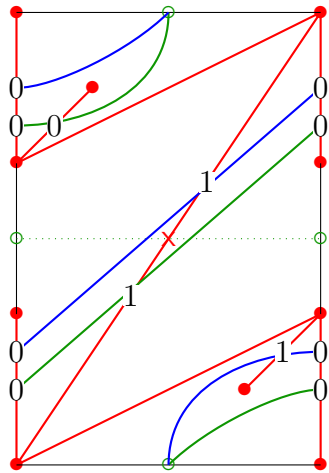
\begin{center}\scalebox{1}{
\begin{tikzpicture}[>=stealth,scale=1]

\node at (0,1) {$\rpoint$};
\node at (0,3) {$\rpoint$};
\node at (4,1) {$\rpoint$};
\node at (1,2) {$\rpoint$};
\node at (4,3) {$\rpoint$};
\node at (2,0) {$\cross$};

\node at (0,-1) {$\rpoint$};
\node at (0,-3) {$\rpoint$};
\node at (4,-1) {$\rpoint$};
\node at (3,-2) {$\rpoint$};
\node at (4,-3) {$\rpoint$};

\draw[red, thick] (0,3)--(0,1)--(4,3)--(4,1);

\draw[red, thick] (0,-1)--(0,-3)--(4,-1)--(4,-3);

\draw[red,thick] (0,1)--(1,2);
\draw[red,thick] (4,-1)--(3,-2);

\draw[red,thick] (4,3)--(0,-3);

\draw (0,3)--(4,3);
\draw (0,1)--(0,-1);
\draw (4,1)--(4,-1);
\draw (0,-3)--(4,-3);

\node at (0,0) {$\gpoint$};
\node at (4,0) {$\gpoint$};

\draw[dark-green,dotted] (0,0)--(4,0);

\node at (2,3) {$\gpoint$};
\node at (2,-3) {$\gpoint$};

\draw[dark-green, thick] (2,-3).. controls (2.5,-2.5) and (3.5,-2).. (4,-2);
\draw[blue, thick] (2,-3).. controls (2,-2) and (3,-1.5)..(4,-1.5);

\draw[dark-green,thick] (0,-2)--(4,1.5);
\draw[blue,thick] (0,-1.5)--(4,2);

\draw[dark-green, thick] (0,1.5).. controls (1,1.5) and (2,2)..(2,3);

\draw[blue,thick] (0,2).. controls (0.5,2) and (1.5,2.5)..(2,3);

\node[fill=white, inner sep= 0pt] at (4,-2) {$0$};
\node[fill=white, inner sep= 0pt] at (0,-2) {$0$};
\node[fill=white, inner sep= 0pt] at (4,2) {$0$};
\node[fill=white, inner sep= 0pt] at (4,1.5) {$0$};
\node[fill=white, inner sep= 0pt] at (4,-1.5) {$0$};
\node[fill=white, inner sep= 0pt] at (0,2) {$0$};
\node[fill=white, inner sep= 0pt] at (0,1.5) {$0$};
\node[fill=white, inner sep= 0pt] at (0,-1.5) {$0$};
\node[fill=white, inner sep= 0pt] at (1.5,-0.75) {$1$};
\node[fill=white, inner sep= 0pt] at (2.5,0.75) {$1$};
\node[fill=white, inner sep= 0pt] at (3.5,-1.5) {$1$};
\node[fill=white, inner sep= 0pt] at (0.5,1.5) {$0$};

\end{tikzpicture}}\end{center}

These two graded curves correspond to the following complexes in $\cD^{\rm b}(\widetilde{A})$:

$$ \xymatrix{P_{1^+}\oplus P_{1^-}\ar[rr]^{\begin{pmatrix} a^+ & a^-\\ 0& c^-\end{pmatrix}} && P_2\oplus P_{4^-}}\textrm{ and }\xymatrix{P_{1^-}\oplus P_{1^+}\ar[rr]^{\begin{pmatrix} a^- & a^+\\ 0 & c^+\end{pmatrix}} && P_2\oplus P_{4^+}}.$$

Their image through the functor $\cD^{\rm b}(\widetilde{A})\to \cD^{\rm b}(A)$ gives the following complexes which are isomorphic:

$$\xymatrix{P_1\oplus P_1\ar[rrr]^{\begin{pmatrix}a & a \\ a' & -a'\\ 0 & c\end{pmatrix}} & && P_2\oplus P_{2'}\oplus P_4},\textrm{ and }\xymatrix{P_1\oplus P_1\ar[rrr]^{\begin{pmatrix}a & a \\ -a' & a'\\ 0 & c\end{pmatrix}} & && P_2\oplus P_{2'}\oplus P_4}.$$

Take now a $(\gamma,\grading)\in\pi_1 ^{\rm orb, gr}(\surf,M_{\gpoint},P_{\rpoint})$ such that $\gamma^2=1$ as follows:

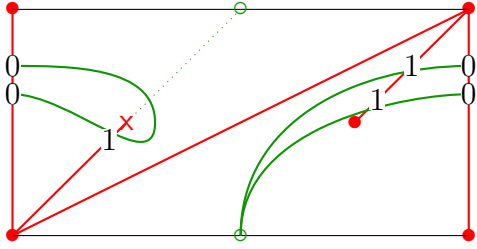
\begin{center}\scalebox{1}{
\begin{tikzpicture}[>=stealth,scale=1.5]

\node at (0,0) {$\rpoint$};
\node at (0,2) {$\rpoint$};
\node at (4,0) {$\rpoint$};
\node at (3,1) {$\rpoint$};
\node at (4,2) {$\rpoint$};
\node at (1,1) {$\cross$};

\draw[red, thick] (0,2)--(0,0)--(4,2)--(4,0);
\draw[red,thick] (0,0)--(1,1);
\draw[red,thick] (4,2)--(3,1);
\draw (0,0)--(4,0);
\draw (0,2)--(4,2);

\node at (2,0) {$\gpoint$};
\node at (2,2) {$\gpoint$};

\draw[dark-green,dotted] (1,1)--(2,2);

\draw[dark-green, thick] (0,1.25).. controls (0.5,1.25) and  (1.25,0.5)..(1.25,1).. controls (1.25,1.5) and (0.5,1.5)..(0,1.5);

\draw[dark-green, thick] (2,0).. controls (2,1) and (3.5,1.25).. (4,1.25);
\draw[dark-green, thick] (2,0).. controls (2,1) and (3,1.5).. (4,1.5);

\node[fill=white, inner sep=0pt] at (0,1.25) {$0$};
\node[fill=white, inner sep=0pt] at (0,1.5) {$0$};
\node[fill=white, inner sep=0pt] at (4,1.25) {$0$};
\node[fill=white, inner sep=0pt] at (4,1.5) {$0$};
\node[fill=white, inner sep=0pt] at (3.5,1.5) {$1$};
\node[fill=white, inner sep=0pt] at (0.85,0.85) {$1$};
\node[fill=white, inner sep=0pt] at (3.2,1.2) {$1$};

\end{tikzpicture}}\end{center}

The graded curve $(\gamma,\grading)$ is in the set $(\surf 2)$ and  has a unique preimage in $\pi_1^{\rm gr}(\widetilde{\surf},\widetilde{M}_{\gpoint},\widetilde{P}_{\rpoint})$. 

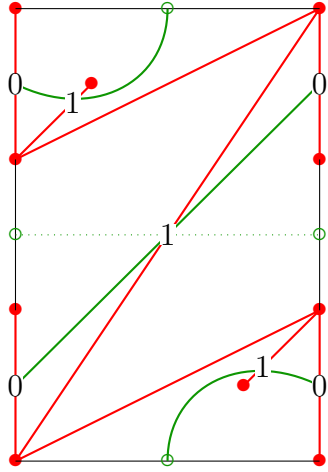
\begin{center}\scalebox{1}{
\begin{tikzpicture}[>=stealth,scale=1]

\node at (0,1) {$\rpoint$};
\node at (0,3) {$\rpoint$};
\node at (4,1) {$\rpoint$};
\node at (1,2) {$\rpoint$};
\node at (4,3) {$\rpoint$};
\node at (2,0) {$\cross$};

\node at (0,-1) {$\rpoint$};
\node at (0,-3) {$\rpoint$};
\node at (4,-1) {$\rpoint$};
\node at (3,-2) {$\rpoint$};
\node at (4,-3) {$\rpoint$};

\draw[red, thick] (0,3)--(0,1)--(4,3)--(4,1);

\draw[red, thick] (0,-1)--(0,-3)--(4,-1)--(4,-3);

\draw[red,thick] (0,1)--(1,2);
\draw[red,thick] (4,-1)--(3,-2);

\draw[red,thick] (4,3)--(0,-3);

\draw (0,3)--(4,3);
\draw (0,1)--(0,-1);
\draw (4,1)--(4,-1);
\draw (0,-3)--(4,-3);

\node at (0,0) {$\gpoint$};
\node at (4,0) {$\gpoint$};

\draw[dark-green,dotted] (0,0)--(4,0);

\node at (2,3) {$\gpoint$};
\node at (2,-3) {$\gpoint$};

\draw[dark-green, thick] (2,-3).. controls (2,-2) and (3,-1.5).. (4,-2);

\draw[dark-green,thick] (0,-2)--(4,2);

\draw[dark-green, thick] (0,2).. controls (1,1.5) and (2,2)..(2,3);

\node[fill=white, inner sep= 0pt] at (4,-2) {$0$};
\node[fill=white, inner sep= 0pt] at (0,-2) {$0$};
\node[fill=white, inner sep= 0pt] at (4,2) {$0$};

\node[fill=white, inner sep= 0pt] at (0,2) {$0$};

\node[fill=white, inner sep= 0pt] at (2,0) {$1$};

\node[fill=white, inner sep= 0pt] at (3.25,-1.75) {$1$};
\node[fill=white, inner sep= 0pt] at (0.75,1.75) {$1$};

\end{tikzpicture}}\end{center}

The corresponding object in $\cD^{\rm b}(\widetilde{A})$ is 
$$\xymatrix{P_{1^+}\oplus P_{1^-}\ar[rrr]^{\begin{pmatrix}c^+ & 0\\ a^+ & a^-\\ 0 & c^-\end{pmatrix}} &&& P_{4^+}\oplus P_2\oplus P_{4^-} }$$
Its image in $\cD^{\rm b}(A)$ is the following complex
$$\xymatrix{P_1\oplus P_1\ar[rrr]^{\begin{pmatrix}c & 0\\ a & a\\ a' & -a'\\ 0 & c \end{pmatrix}}&&&P_4\oplus P_2\oplus P_{2'}\oplus P_4},$$
which is isomorphic to the complex

$$\xymatrix{P_1\oplus P_1\ar[rrr]^{\begin{pmatrix}c & 0\\ a & 0\\ 0 & a'\\ 0 & c \end{pmatrix}}&&&P_4\oplus P_2\oplus P_{2'}\oplus P_4}$$
which clearly decomposes into the sum of two indecomposable complexes. 

\bigskip

Let  $(\gamma,\grading)\in\pi_1 ^{\rm orb, gr}(\surf,M_{\gpoint},P_{\rpoint})$ be as follows, where the curve $\gamma$ has infinite rays circle around the point in $P_\rpoint$, and where $\gamma =\gamma^{-1}$. :

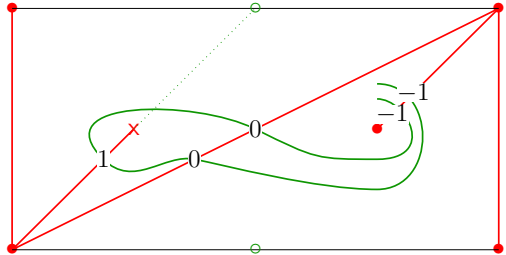
\begin{center}\scalebox{0.8}{
\begin{tikzpicture}[>=stealth,scale=2]

\node at (0,0) {$\rpoint$};
\node at (0,2) {$\rpoint$};
\node at (4,0) {$\rpoint$};
\node at (3,1) {$\rpoint$};
\node at (4,2) {$\rpoint$};
\node at (1,1) {$\cross$};

\draw[red, thick] (0,2)--(0,0)--(4,2)--(4,0);
\draw[red,thick] (0,0)--(1,1);
\draw[red,thick] (4,2)--(3,1);
\draw (0,0)--(4,0);
\draw (0,2)--(4,2);

\node at (2,0) {$\gpoint$};
\node at (2,2) {$\gpoint$};

\draw[dark-green,dotted] (1,1)--(2,2);

\draw[dark-green, thick] (3,1.375)..controls (3.5,1.375) and (3.5,0.5)..(3,0.5).. controls (2.5,0.5) and (1.75,0.7).. (1.5,0.75).. controls (1.25,0.8) and  (1,0.5)..(0.75,0.75).. controls (0.25,1.25) and  (1.5,1.25).. (2,1).. controls (2.5,0.75) and (2.5, 0.75).. (3,0.75).. controls (3.5, 0.75) and (3.25,1.25).. (3,1.25) ;

\node[fill=white, inner sep=0pt] at (3.125,1.125) {$-1$};
\node[fill=white, inner sep=0pt] at (3.3,1.3) {$-1$};
\node[fill=white, inner sep=0pt] at (2,1) {$0$};
\node[fill=white, inner sep=0pt] at (1.5,0.75) {$0$};
\node[fill=white, inner sep=0pt] at (0.75,0.75) {$1$};

\end{tikzpicture}}\end{center}

The graded curve $(\gamma, \grading)$ has a unique preimage in $\pi_1^{\gr}(\widetilde{\surf},\widetilde{M}_{\gpoint},\widetilde{P}_{\rpoint})$ which is as follows :

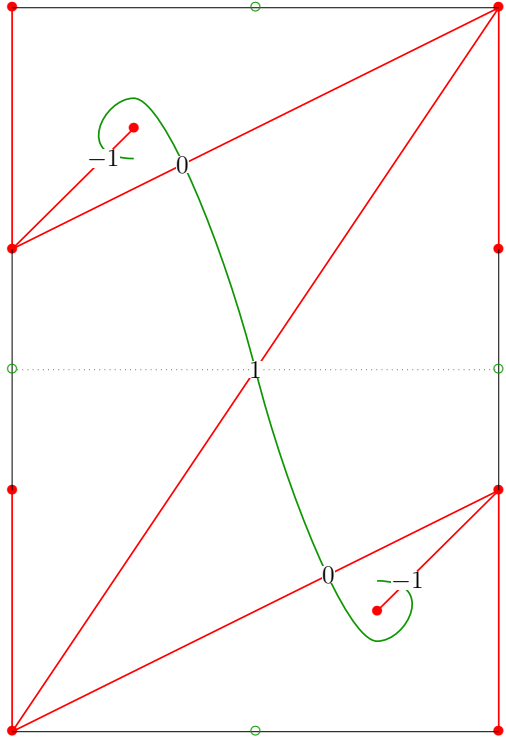
\begin{center}\scalebox{0.8}{
\begin{tikzpicture}[>=stealth,scale=2]

\node at (0,1) {$\rpoint$};
\node at (0,3) {$\rpoint$};
\node at (4,1) {$\rpoint$};
\node at (1,2) {$\rpoint$};
\node at (4,3) {$\rpoint$};
\node at (2,0) {$\cross$};

\node at (0,-1) {$\rpoint$};
\node at (0,-3) {$\rpoint$};
\node at (4,-1) {$\rpoint$};
\node at (3,-2) {$\rpoint$};
\node at (4,-3) {$\rpoint$};

\draw[red, thick] (0,3)--(0,1)--(4,3)--(4,1);

\draw[red, thick] (0,-1)--(0,-3)--(4,-1)--(4,-3);

\draw[red,thick] (0,1)--(1,2);
\draw[red,thick] (4,-1)--(3,-2);

\draw[red,thick] (4,3)--(0,-3);

\draw (0,3)--(4,3);
\draw (0,1)--(0,-1);
\draw (4,1)--(4,-1);
\draw (0,-3)--(4,-3);

\node at (0,0) {$\gpoint$};
\node at (4,0) {$\gpoint$};

\draw[dark-green,dotted] (0,0)--(4,0);

\node at (2,3) {$\gpoint$};
\node at (2,-3) {$\gpoint$};

\draw[dark-green, thick] (1,1.75).. controls (0.5,1.75) and (0.75,2.25)..(1,2.25).. controls (1.25,2.25) and (1.75,1).. (2,0).. controls (2.25,-1) and (2.75,-2.25)..(3,-2.25).. controls (3.25,-2.25) and (3.5,-1.75)..(3,-1.75) ;

\node[fill=white, inner sep=0pt] at (1.4,1.7) {$0$};
\node[fill=white, inner sep=0pt] at (2.6,-1.7) {$0$};
\node[fill=white, inner sep=0pt] at (2,0) {$1$};
\node[fill=white, inner sep=0pt] at (0.75,1.75) {$-1$};
\node[fill=white, inner sep=0pt] at (3.25,-1.75) {$-1$};

\end{tikzpicture}}\end{center}

The corresponding object in $\cD^{\rm b}(\widetilde{A})$ is the following complex (infinite on the left):

$$\xymatrix@C=1.5cm{\cdots  P_{4^+}\oplus P_{4^{-}}\ar[r]^{\begin{pmatrix} e^+ & 0\\ 0 & e^-\end{pmatrix}} &P_{4^+}\oplus P_{4^{-}}\ar[r]^-{\begin{pmatrix} e^+ & 0\\ 0 & 0 \\ 0 & e^-\end{pmatrix}} & P_{4^+}\oplus P_2\oplus P_{4^{-}}\ar[rr]^-{\begin{pmatrix} d^+e^+ & b^+ & 0\\ 0 & b^- & d^-e^- \end{pmatrix}} && P_{3^+}\oplus P_{3^-} }$$

Its image in $\cD^{\rm b}(A)$ is the following complex

$$\xymatrix@C=1.5cm{\cdots  P_{4}^2 \ar[r]^{\begin{pmatrix} e & 0\\ 0 & e\end{pmatrix}} &P_4^2\ar[r]^-{\begin{pmatrix} e & 0\\ 0 & 0\\ 0 & 0 \\ 0 & e\end{pmatrix}} & P_{4}\oplus P_2^2\oplus P_{4}\ar[rr]^-{\begin{pmatrix} de & b & b' & 0\\ 0 & b & -b' & de \end{pmatrix}} && P_{3}^2}$$
which is isomorphic to the complex 

$$\xymatrix@C=1.5cm{\cdots  P_{4}^2 \ar[r]^{\begin{pmatrix} e & 0\\ 0 & e\end{pmatrix}} &P_4^2\ar[r]^-{\begin{pmatrix} e & 0\\ 0 & 0\\ 0 & 0 \\ 0 & e\end{pmatrix}} & P_{4}\oplus P_2^2\oplus P_{4}\ar[rr]^-{\begin{pmatrix} de & b & 0 & 0\\ 0 & 0 & b' & de \end{pmatrix}} && P_{3}^2}$$
which decomposes into the sum of two indecomposable summands.

\subsection{Objects in the sets $(\surf 3)$, $(\surf 4)$ and $(\surf 5)$}
Now let $([\gamma],\grading,\lambda)\in \pi_1^{\rm orb,free,gr}(\surf)\times k^*$ be the following graded curve

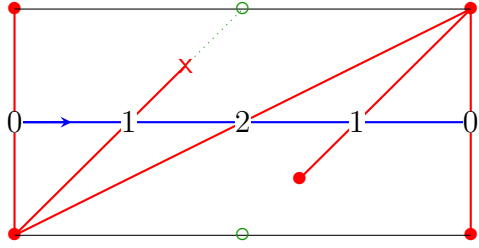
\begin{center}\scalebox{1}{
\begin{tikzpicture}[>=stealth,scale=1.5]

\node at (0,0) {$\rpoint$};
\node at (0,2) {$\rpoint$};
\node at (4,0) {$\rpoint$};
\node at (2.5,0.5) {$\rpoint$};
\node at (4,2) {$\rpoint$};
\node at (1.5,1.5) {$\cross$};

\draw[red, thick] (0,2)--(0,0)--(4,2)--(4,0);
\draw[red,thick] (0,0)--(1.5,1.5);
\draw[red,thick] (4,2)--(2.5,0.5);
\draw (0,0)--(4,0);
\draw (0,2)--(4,2);

\node at (2,0) {$\gpoint$};
\node at (2,2) {$\gpoint$};

\draw[dark-green,dotted] (1.5,1.5)--(2,2);

\draw[blue, thick] (0,1)--(4,1);
\draw[blue,thick,->] (0,1)--(0.5,1);

\node[fill=white, inner sep=0pt] at (0,1) {$0$};
\node[fill=white, inner sep=0pt] at (1,1) {$1$};
\node[fill=white, inner sep=0pt] at (3,1) {$1$};
\node[fill=white, inner sep=0pt] at (2,1) {$2$};
\node[fill=white, inner sep=0pt] at (4,1) {$0$};

\end{tikzpicture}}\end{center}

The element $([\gamma],\grading,\lambda)$ is the set $(\surf 3)$ and $[\gamma]$ is in the image of $\Psi$ (indeed it intersects the green dotted lines an even number of times).

The graded curve $([\gamma],\grading)$ has two preimages in $\pi_1^{\rm free,gr}(\widetilde{\surf})$ (that are in the set $(\widetilde{\surf}3)$ which are as follows:

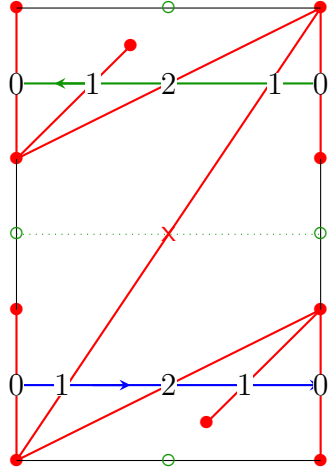
\begin{center}\scalebox{1}{
\begin{tikzpicture}[>=stealth,scale=1]

\node at (0,1) {$\rpoint$};
\node at (0,3) {$\rpoint$};
\node at (4,1) {$\rpoint$};
\node at (1.5,2.5) {$\rpoint$};
\node at (4,3) {$\rpoint$};
\node at (2,0) {$\cross$};

\node at (0,-1) {$\rpoint$};
\node at (0,-3) {$\rpoint$};
\node at (4,-1) {$\rpoint$};
\node at (2.5,-2.5) {$\rpoint$};
\node at (4,-3) {$\rpoint$};

\draw[red, thick] (0,3)--(0,1)--(4,3)--(4,1);
\draw[red, thick] (0,-1)--(0,-3)--(4,-1)--(4,-3);

\draw[red,thick] (0,1)--(1.5,2.5);
\draw[red,thick] (4,-1)--(2.5,-2.5);

\draw[red,thick] (4,3)--(0,-3);

\draw (0,3)--(4,3);
\draw (0,1)--(0,-1);
\draw (4,1)--(4,-1);
\draw (0,-3)--(4,-3);

\node at (0,0) {$\gpoint$};
\node at (4,0) {$\gpoint$};

\draw[dark-green,dotted] (0,0)--(4,0);

\node at (2,3) {$\gpoint$};
\node at (2,-3) {$\gpoint$};

\draw[blue,thick,->] (0,-2)--(4,-2);
\draw[dark-green, thick] (0,2)--(4,2);

\draw[dark-green,thick,->] (1,2)--(0.5,2);
\draw[blue,thick,->] (1,-2)--(1.5,-2);

\node[fill=white, inner sep= 0pt] at (4,-2) {$0$};
\node[fill=white, inner sep= 0pt] at (0,-2) {$0$};
\node[fill=white, inner sep= 0pt] at (4,2) {$0$};
\node[fill=white, inner sep= 0pt] at (0,2) {$0$};

\node[fill=white, inner sep= 0pt] at (1,2) {$1$};
\node[fill=white, inner sep= 0pt] at (2,2) {$2$};
\node[fill=white, inner sep= 0pt] at (3.4,2) {$1$};

\node[fill=white, inner sep= 0pt] at (0.6,-2) {$1$};
\node[fill=white, inner sep= 0pt] at (2,-2) {$2$};
\node[fill=white, inner sep= 0pt] at (3,-2) {$1$};

\end{tikzpicture}}\end{center}

These graded curves correspond respectively to the following objects in $\cD^{\rm b}(\widetilde{A})$:

$$\xymatrix{P_{1^+}\ar[r]^-{\begin{pmatrix}\lambda a^+\\ d^+\end{pmatrix}} & P_2\oplus P_{4^+} \ar[rr]^{\begin{pmatrix}b^+ & c^+\end{pmatrix}}& &P_{3^+}}\textrm{ and }\xymatrix{P_{1^-}\ar[r]^-{\begin{pmatrix}\lambda a^-\\ d^-\end{pmatrix}} & P_2\oplus P_{4^-} \ar[rr]^{\begin{pmatrix}b^- & c^-\end{pmatrix}}& &P_{3^-}}$$

The corresponding complexes in $\cD^{\rm b}(A)$ are 
$$\xymatrix{P_1\ar[r]^-{\begin{pmatrix}\lambda a\\ \lambda a'\\ d\end{pmatrix}} & P_2\oplus P_{2'}\oplus P_4\ar[rr]^-{\begin{pmatrix}b & b'& c\end{pmatrix}} & & P_3}\textrm{ and }\xymatrix{P_1\ar[r]^-{\begin{pmatrix}\lambda a\\ -\lambda a'\\ d\end{pmatrix}} & P_2\oplus P_{2'}\oplus P_4\ar[rr]^-{\begin{pmatrix}b & -b'& c\end{pmatrix}} & & P_3}$$
which are isomorphic.

\medskip

Let $([\gamma],\grading,\lambda)\in \pi_1^{\rm orb,free,gr}(\surf)\times k^*$ be the following graded curve. It is in the set $(\surf 3)$ and $[\gamma]$ is not the image of $\Psi$ since it intersects the green dotted lines an odd number of times.

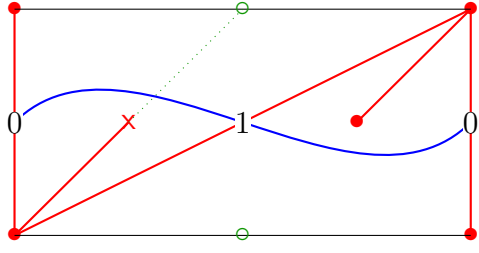
\begin{center}\scalebox{1}{
\begin{tikzpicture}[>=stealth,scale=1.5]

\node at (0,0) {$\rpoint$};
\node at (0,2) {$\rpoint$};
\node at (4,0) {$\rpoint$};
\node at (3,1) {$\rpoint$};
\node at (4,2) {$\rpoint$};
\node at (1,1) {$\cross$};

\draw[red, thick] (0,2)--(0,0)--(4,2)--(4,0);
\draw[red,thick] (0,0)--(1,1);
\draw[red,thick] (4,2)--(3,1);
\draw (0,0)--(4,0);
\draw (0,2)--(4,2);

\node at (2,0) {$\gpoint$};
\node at (2,2) {$\gpoint$};

\draw[dark-green,dotted] (1,1)--(2,2);
\draw[blue, thick] (0,1)..controls (1,2) and (3,0)..(4,1);

\node[fill=white, inner sep=0pt] at (0,1) {$0$};
\node[fill=white, inner sep=0pt] at (2,1) {$1$};
\node[fill=white, inner sep=0pt] at (4,1) {$0$};

\end{tikzpicture}}\end{center}

However, the concatenation of its two preimages is in the set $(\widetilde{\surf} 3')$, and is a primitive closed curve as follows:

 \begin{center}\scalebox{1}
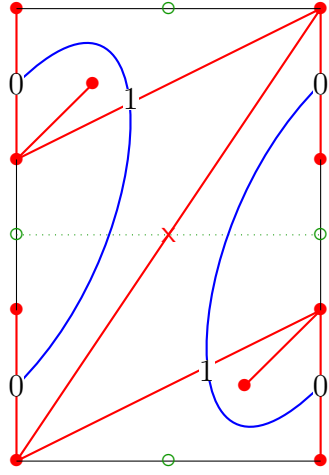
{
\begin{tikzpicture}[>=stealth,scale=1]

\node at (0,1) {$\rpoint$};
\node at (0,3) {$\rpoint$};
\node at (4,1) {$\rpoint$};
\node at (1,2) {$\rpoint$};
\node at (4,3) {$\rpoint$};
\node at (2,0) {$\cross$};

\node at (0,-1) {$\rpoint$};
\node at (0,-3) {$\rpoint$};
\node at (4,-1) {$\rpoint$};
\node at (3,-2) {$\rpoint$};
\node at (4,-3) {$\rpoint$};

\draw[red, thick] (0,3)--(0,1)--(4,3)--(4,1);

\draw[red, thick] (0,-1)--(0,-3)--(4,-1)--(4,-3);

\draw[red,thick] (0,1)--(1,2);
\draw[red,thick] (4,-1)--(3,-2);

\draw[red,thick] (4,3)--(0,-3);

\draw (0,3)--(4,3);
\draw (0,1)--(0,-1);
\draw (4,1)--(4,-1);
\draw (0,-3)--(4,-3);

\node at (0,0) {$\gpoint$};
\node at (4,0) {$\gpoint$};

\draw[dark-green,dotted] (0,0)--(4,0);

\node at (2,3) {$\gpoint$};
\node at (2,-3) {$\gpoint$};

\draw[blue,thick] (0,2).. controls (2,4) and (2,0).. (0,-2);

\draw[blue,thick] (4,-2).. controls (2,-4) and (2,0).. (4,2);

\node[fill=white, inner sep= 0pt] at (0,2) {$0$};
\node[fill=white, inner sep= 0pt] at (1.5,1.8) {$1$};
\node[fill=white, inner sep= 0pt] at (0,-2) {$0$};
\node[fill=white, inner sep= 0pt] at (4,-2) {$0$};
\node[fill=white, inner sep= 0pt] at (2.5,-1.8) {$1$};
\node[fill=white, inner sep= 0pt] at (4,2) {$0$};

\end{tikzpicture}}\end{center}

The corresponding band object in $\cD^{\rm b}(\widetilde{A})$ is given by 
$$\xymatrix{P_{1^+}\oplus P_{1^-}\ar[rrrr]^{\begin{pmatrix}\lambda b^-a^+ & d^-e^-c^-\\ d^+e^+c^+ & b^+a^-\end{pmatrix}} &&&& P_{3^-}\oplus P_{3^+}},$$
whose image in $\cD^{\rm b}(A)$ is 

$$\xymatrix{P_1\oplus P_1\ar[rrrr]^{\begin{pmatrix}\lambda(ba-b'a') & dec\\ dec & ba-b'a' \end{pmatrix}} & && & P_3\oplus P_3}$$ which can be shown to be isomorphic to 

$$\left(\xymatrix{P_1\ar[rrr]^{\lambda'(ba-b'a') +dec} && & P_3}\right)\oplus \left(\xymatrix{P_1\ar[rrr]^{-\lambda'(ba-b'a') +dec} && & P_3} \right)$$
 where $(\lambda')^2=\lambda$.
 
 \medskip

Finally, let $([\gamma],\grading)\in \pi_1^{\rm orb, free,gr}(\surf)$ be such that $[\gamma]=[\gamma^{-1}]$ is as follows

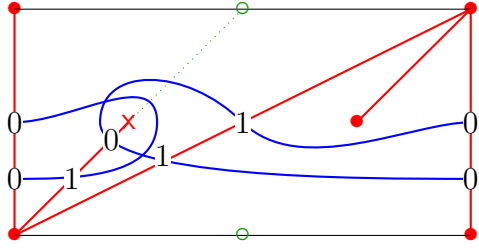
\begin{center}\scalebox{1}{
\begin{tikzpicture}[>=stealth,scale=1.5]

\node at (0,0) {$\rpoint$};
\node at (0,2) {$\rpoint$};
\node at (4,0) {$\rpoint$};
\node at (3,1) {$\rpoint$};
\node at (4,2) {$\rpoint$};
\node at (1,1) {$\cross$};

\draw[red, thick] (0,2)--(0,0)--(4,2)--(4,0);
\draw[red,thick] (0,0)--(1,1);
\draw[red,thick] (4,2)--(3,1);
\draw (0,0)--(4,0);
\draw (0,2)--(4,2);

\node at (2,0) {$\gpoint$};
\node at (2,2) {$\gpoint$};

\draw[dark-green,dotted] (1,1)--(2,2);

\draw[blue, thick] (0,1).. controls (0.5,1) and  (1.25,1.5)..(1.25,1).. controls (1.25,0.5) and (0.5,0.5)..(0,0.5);

\draw[blue,thick] (4,1).. controls (3.5,1) and (2.5,0.5).. (2,1)..controls (1.5,1.5) and (0.75,1.5)..(0.75,1).. controls (0.75,0.5) and (3,0.5).. (4,0.5);

\node[fill=white, inner sep=0pt] at (0,1) {$0$};
\node[fill=white, inner sep=0pt] at (4,1) {$0$};

\node[fill=white, inner sep=0pt] at (2,1) {$1$};
\node[fill=white, inner sep=0pt] at (0.5,0.5) {$1$};
\node[fill=white, inner sep=0pt] at (0,0.5) {$0$};
\node[fill=white, inner sep=0pt] at (0.85,0.85) {$0$};
\node[fill=white, inner sep=0pt] at (1.3,0.7) {$1$};
\node[fill=white, inner sep=0pt] at (4,0.5) {$0$};

\end{tikzpicture}}\end{center}

The closed curve $[\gamma]$ is in the image of $\Psi$ and its preimage is unique as follows:

 \begin{center}\scalebox{1}
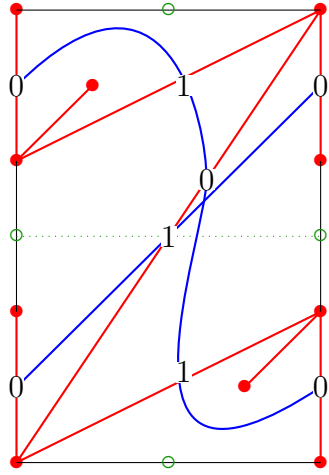
{
\begin{tikzpicture}[>=stealth,scale=1]

\node at (0,1) {$\rpoint$};
\node at (0,3) {$\rpoint$};
\node at (4,1) {$\rpoint$};
\node at (1,2) {$\rpoint$};
\node at (4,3) {$\rpoint$};
\node at (2,0) {$\cross$};

\node at (0,-1) {$\rpoint$};
\node at (0,-3) {$\rpoint$};
\node at (4,-1) {$\rpoint$};
\node at (3,-2) {$\rpoint$};
\node at (4,-3) {$\rpoint$};

\draw[red, thick] (0,3)--(0,1)--(4,3)--(4,1);
\draw[red, thick] (0,-1)--(0,-3)--(4,-1)--(4,-3);
\draw[red,thick] (0,1)--(1,2);
\draw[red,thick] (4,-1)--(3,-2);
\draw[red,thick] (4,3)--(0,-3);

\draw (0,3)--(4,3);
\draw (0,1)--(0,-1);
\draw (4,1)--(4,-1);
\draw (0,-3)--(4,-3);

\node at (0,0) {$\gpoint$};
\node at (4,0) {$\gpoint$};

\draw[dark-green,dotted] (0,0)--(4,0);

\node at (2,3) {$\gpoint$};
\node at (2,-3) {$\gpoint$};

\draw[blue,thick] (0,-2)--(4,2);
\draw[blue,thick] (0,2).. controls (2,4) and (2.5,1.5)..(2.5,0.75)..controls (2.5,0) and (1,-4).. (4,-2);

\node[fill=white, inner sep= 0pt] at (0,2) {$0$};
\node[fill=white, inner sep= 0pt] at (2.2,2) {$1$};
\node[fill=white, inner sep= 0pt] at (2.5,0.75) {$0$};
\node[fill=white, inner sep= 0pt] at (2.2,-1.8) {$1$};
\node[fill=white, inner sep= 0pt] at (4,-2) {$0$};
\node[fill=white, inner sep= 0pt] at (0,-2) {$0$};
\node[fill=white, inner sep= 0pt] at (2,0) {$1$};
\node[fill=white, inner sep= 0pt] at (4,2) {$0$};

\end{tikzpicture}}\end{center}
 The corresponding complex in $\cD^{\rm b}(\widetilde{A})$ is the following
 $$\xymatrix{P_{1^+}\oplus P_{1^-}\oplus P_2\ar[rrrr]^{\begin{pmatrix}\lambda a^+ & a^- & 0\\0 &  d^-e^-c^- & b^-\\ d^+e^+c^+ & 0& b^+\end{pmatrix}}&&&& P_2\oplus P_{3^-}\oplus P_{3^+}}.$$

Its image in $\cD^{\rm b}(A)$ is the following complex
$$\xymatrix{P_1^2\oplus P_2\oplus P_{2'}\ar[rrrrr]^{\begin{pmatrix}\lambda a & a & 0 & 0\\ \lambda a' & -a' & 0 & 0\\ 0 & dec & b & -b'\\ dec & 0 & b & b' \end{pmatrix}} &&&&& P_{2}\oplus P_{2'}\oplus P_3^2}$$

For $\lambda\neq\pm 1$ (so in the case where $([\gamma],\grading,\lambda)$ is in the set $(\surf 4)$), this complex is indecomposable. 

\medskip

For $\lambda= 1$ this complex is isomorphic to 
$$\xymatrix{P_1^2\oplus P_2\oplus P_{2'}\ar[rrrrr]^{\begin{pmatrix} a & 0 & 0 & 0\\ 0  & a' & 0 & 0\\ 0 & dec & 0 & b'\\ dec & 0 & b & 0 \end{pmatrix}} &&&&& P_{2}\oplus P_{2'}\oplus P_3^2}$$
which decomposes.

For $\lambda=-1$ this complex is isomorphic to 
$$\xymatrix{P_1^2\oplus P_2\oplus P_{2'}\ar[rrrrr]^{\begin{pmatrix} a & 0 & 0 & 0\\ 0  & a' & 0 & 0\\ 0 & dec & b & 0\\ dec & 0 & 0 & b' \end{pmatrix}} &&&&& P_{2}\oplus P_{2'}\oplus P_3^2}$$

which also decomposes. Hence we obtain four non isomorphic indecomposable corresponding to the element $([\gamma],\grading)\in\pi_1^{\rm orb,free,gr}(\surf)$ as mentionned in the set $(\surf 5)$.

\end{document}